\documentclass[10pt, a4]{article}

\usepackage[utf8]{inputenc}
\usepackage{amsmath}
\usepackage{mathtools}
\usepackage{amsthm}
\usepackage{amssymb}
\usepackage{amsfonts}
\usepackage[all,cmtip]{xy}
\usepackage{todonotes}
\usepackage{appendix}
\usepackage{thmtools}
\usepackage{thm-restate}

\newcommand{\Tmf}{\text{\textit{Tmf}}}
\newtheorem{defi}{Definition}
\newtheorem{prop}{Proposition}

\newtheorem{coro}{Corollary}
\newtheorem{theo}{Theorem}

\newtheorem{lemma}{Lemma}[]

\theoremstyle{definition}
\newtheorem{remark}{Remark}
\newtheorem{example}{Example}
\newtheorem{convention}{Convention}
\title{Strictly Commutative Complex Orientations of $Tmf_1(N)$}
\author{Dominik Absmeier}
\date{}
\begin{document}

\maketitle
\begin{abstract}
We construct by adapting methods and results of Ando, Hopkins, Rezk and Wilson combined with results of Hopkins and Lawson strictly commutative complex orientations for the spectra of topological modular forms with level $\Gamma_1(N)$.
\end{abstract}
\newcommand{\Addresses}{{
  \bigskip
  \footnotesize

\textsc{Fakult\"at f\"ur Mathematik, Ruhr-Universit\"at Bochum,
    IB3/173, 44801 Bochum, Germany}\par\nopagebreak
  \textit{E-mail address:} \texttt{dominik.absmeier@ruhr-uni-bochum.de} }}
\section{Introduction}
Ever since Quillen's theorem in the 1970s showing that the Hopf algebroid $(MU_*,MU_*MU)$ represents one-dimensional formal group laws and strict isomorphism between those and the subsequent work of Mahowald, Ravenel, Wilson, Morava, Hopkins and many others who exploited this result to the fullest it has become clear that the world of stable homotopy is intimately related to the world of smooth one dimensional formal groups.
A natural source for formal groups is provided by abelian group schemes $A$ who give rise to a formal group $\widehat{A}$ by completion along the identity section. 
As the formal group of an $n$-dimensional abelian group scheme is itself $n$-dimensional one has to either look at specific abelian group schemes with additional structures which allow one to split off an 1-dimensional summand of its formal group (see \cite{behrenslawson}) or one does the obvious thing first and looks at 1-dimensional abelian group schemes. There are basically two classes of those: \begin{enumerate}
    \item  The multiplicative group $\mathbb{G}_m$ whose formal group $\widehat{\mathbb{G}}_m$ is realised by a very well known complex oriented cohomology theory, complex K-theory.
   \item Elliptic curves. Contrary to the case of the multiplicative group, the formal group of an elliptic curve can have height 1 or 2 and thus could permit detection of elements on the 2-line of the Adams-Novikov spectral sequence.
\end{enumerate}
A complex orientable cohomology theory whose formal group is isomorphic to the formal group of an elliptic curve is called an elliptic cohomology theory. 
The first example of such an elliptic cohomology theory was studied by Landweber, Ochanine, Ravenel and Stong (see \cite{landweberravenelstong}) who constructed, using the Landweber exact funtor theorem, a generalized cohomology theory whose formal group is the formal group of the universal elliptic curve with a so called $\Gamma_0(2)$-level structure. 

It turns out that in general the Landweber exact functor theorem furnishes many  examples of elliptic cohomology theories. Elliptic curves are classified by a certain Deligne-Mumford stack $\mathcal{M}_{ell}$ which we will revisit in Section 2. The map $\mathcal{M}_{ell} \rightarrow \mathcal{M}_{fg}$ to the moduli stack of formal groups which associates to every elliptic curve its corresponding formal group turns out to be flat. A modern reformulation of the Landweber exact functor theorem (see \cite{tmfbook}, Chapter 4)
states
that for every flat map 
$$\text{Spec}(R) \rightarrow \mathcal{M}_{fg}$$ 
one obtains an even periodic complex orientable cohomology theory $E$ with coefficients $\pi_{even}E=R$.
In particular every elliptic curve $C$ over a ring $R$ classified by a flat morphism 
$$\text{Spec}(R) \rightarrow \mathcal{M}_{ell}
$$
gives rise to an even periodic elliptic cohomology theory $E$ with coefficients $\pi_{2*} E = R$ and associated formal group $\widehat{C}$. This suggested the possibility of equipping the stack $\mathcal{M}_{ell}$ with a sheaf of even periodic cohomology theories, allowing for a unified way to study elliptic cohomology theories. This idea was put on a concrete footing by Goerss, Hopkins and Miller who had the crucial insight that in order to construct such a (pre)sheaf one has to work in the world of $\mathbb{E}_{\infty}$-ring spectra, that is ring spectra which are not only commutative up to homotopy but already "on the nose" on the level of spectra. To be able to perform all the required constructions it is paramount to have a category of spectra with a symmetric monoidal smash product and it was only in the 1990s - more than 30 years after the inception of the notion of a spectrum that such strict models for the category of spectra were available through the work of Elmendorf-Kriz-Mandell-May (\cite{ekmms}, Hovey-Shipley-Smith (\cite{hss} and Mandell-May-Schwede-Shipley (\cite{mmss}).
After the technical tools were at their disposal, Goerss, Hopkins and Miller were able to construct a presheaf $\mathcal{O}^{top}$ of $\mathbb{E}_{\infty}$-ring spectra on the \'etale site of the moduli stack $\overline{\mathcal{M}}_{ell}$ of generalized elliptic curves (for the relation to classical elliptic curves see Section 2). Once this presheaf was at hand one could construct a \lq\lq universal elliptic cohomology theory" by taking global sections. This is the spectrum which nowadays goes by the name $Tmf$ or topological modular forms.
It turns out while every elliptic cohomology theory is complex orientable, the spectrum $Tmf$ itself is not. The reason is roughly speaking due to the fact that elliptic curves can have nontrivial automorphisms at the primes $2$ and $3$ which act on the corresponding elliptic cohomology theories. As $Tmf$ in a sense encompasses \emph{all} elliptic cohomology theories at once the existence of these automorphisms destroys complex orientability for $Tmf$ (note however that  $Tmf$ becomes complex orientable after inverting $6$). 

Although there is no map of ring spectra $MU\rightarrow Tmf$ there were two pieces of evidence which suggested that $Tmf$ should come with an orientation from another Thom spectrum, namely $MString$ (often also called $MO \langle 8 \rangle$) which is the Thom spectrum corresponding to the connective cover $BO\langle 8 \rangle$ of the classifying space for the orthogonal group:
\begin{enumerate}
    \item The mod 2 homology of $Tmf$ is given by a certain module $\mathcal{A}//\mathcal{A}(2)$ over the dual of the Steenrod algebra which was already known to be a direct summand in the homology of the spectrum $MString$. We would like to point out that it is an open question if this algebraic splitting lifts to a splitting on the spectrum level. For some recent progress on this question see \cite{lauresschuster}.
    \item In his work on the Dirac operator on loop spaces \cite{witten} Witten constructed a genus on string manifolds which takes values in the (integral) ring of modular forms. Now genera often lift to orientations like for {example} in the case of the $\widehat{A}$-genus which lifts to a map of ring spectra
    $$MSpin \rightarrow KO.
    $$
    Earlier Ochanine (\cite{ochanine}) constructed a related genus on the bordism ring of orientable manifolds which takes values in the graded ring $MF(\Gamma_0(2))$ of modular forms with $\Gamma_0(2)$-level structures (for more on modular forms and level structures see Section 2). This genus lifts to a map of ring spectra with target the elliptic cohomology theory of Landweber, Ravenel, Stong mentioned at beginning of this introduction.
    As the torsion free part of $Tmf$ coincides with the graded ring of integral modular forms it was natural to suspect that $Tmf$ admits an orientation by $MString$.
\end{enumerate}
Laures \cite{gerdwitten} showed in 2004 that $K(1)$-locally at the prime $2$ such a topological lift of the Witten genus exists and some years later Ando, Hopkins and Rezk \cite{AHR} consructed an integral lift of the Witten genus
$$MString \rightarrow Tmf.
$$
As we shall discuss in Section 2 elliptic curves can be equipped with additional structures, so called level structures $\Gamma$. For example a $\Gamma_1(N)$-level structure on an elliptic curve $E$ corresponds to a designated point $P$ on $E$ of exact order $N$. Generalized elliptic curves with such extra structure are again classified by Deligne-Mumford stacks $\overline{\mathcal{M}_{ell}}(\Gamma)$.
Hill and Lawson showed in \cite{hilllawson} that one can extend the presheaf $\mathcal{O}^{top}$ to these moduli stacks and thus again, after passing to global sections, obtain $\mathbb{E}_{\infty}$-ring spectra $Tmf(\Gamma)$ - topological modular forms with level $\Gamma$. They asked if these spectra also admit $\mathbb{E}_{\infty}$-orientations by Thom spectra other then $MString$ (there are canonical $\mathbb{E}_{\infty}$ maps $Tmf \rightarrow Tmf(\Gamma)$ which of course give rise to $MString$ orientations by precomposing with the topological lift of the Witten genus). In \cite{Dylan} Wilson gave a partial positive answer and constructed  $\mathbb{E}_{\infty}$ orientations 
$$ MSpin[\frac{1}{2}] \rightarrow Tmf(\Gamma_0(2))
$$

and 
$$ MSpin[\frac{1}{6}] \rightarrow Tmf[\frac{1}{6}]
$$
by adapting the methods of \cite{AHR}.
Furthermore he gave a complete characterization of $\mathbb{E}_{\infty}$ $MString$-orientations for level $\Gamma_0(N)$ in terms of $p$-adic measures. 
It turns out that for specific level $\Gamma$ the theories $Tmf(\Gamma)$ are complex orientable through \emph{homotopy commutative} ring maps 
$$MU \rightarrow Tmf(\Gamma),
$$ 
and in fact the first examples of these theories $Tmf(\Gamma)$ as homotopy commutative ring spectra were constructed this way via the exact functor theorem (see \cite{topqexpansion}).

One can ask wether any of these orientations lift to strictly commutative ones. While there is an abundance of complex orientable cohomology theories, many of which are known to admit an $\mathbb{E}_{\infty}$-ring structure, in general our knowledge about strictly commutative complex orientations is rather limited. We would like to summarize what is known so far: In his thesis \cite{ando} (see also \cite{ando2}) Ando gave a criterion for complex orientations to refine to the slightly weaker structure of an $\mathbb{H}_{\infty}$-map in terms of the coordinate of the corresponding formal group law. Furthermore he showed that the Lubin-Tate theories $E_n$ corresponding to the Honda formal group admit such an orientation. Joachim showed in \cite{joachim}, using operator algebra techniques, that the Atiyah-Bott-Shapiro orientation $$ABS:MSpin^c\rightarrow KU$$ 
of complex K-theory refines to an $\mathbb{E}_{\infty}$-map and hence, as there is a canonical map of Thom spectra 
$$MU\rightarrow MSpin^c, $$ providing the first nontautological example of a strictly commutative complex orientation.
Then Hopkins and Lawson described in \cite{hopkinslawson} a general approach to constructing complex $\mathbb{E}_{\infty}$-orientations and showed that in the case of $K(1)$-local $\mathbb{E}_{\infty}$-ring spectra the aforementioned criterion of Ando is sufficient and necessary for a complex orientation to actually lift to an $\mathbb{E}_{\infty}$-map (not just $\mathbb{H}_{\infty}$). We should mention that this result has been previously announced for $p>2$ in \cite{walker} for the case of $p$-completed K-theory and \cite{jan} in the general case. Unfortunately the arguments in both works contain a small gap and the only way we are aware of to fix these gaps is by applying results of \cite{hopkinslawson} very much in the spirit of the methods of Section 5.4. Finally, in \cite{zhu} Zhu showed that every Lubin-Tate spectrum admits an $\mathbb{H}_{\infty}$-orientation and hence in light of the work of Hopkins and Lawson thus also provided the first family of strictly commutative complex orientable spectra. Note that so far nothing is known beyond chromatic height 1. This work aims to remedy this.
\begin{restatable}{theo}{orientations}
\label{msuorientations}
Let $\chi$ be a nontrivial Dirichlet character modulo $N$ such that $\chi(-1)=-1$ and $G_{k}^{\chi}$ the corresponding Eisenstein series.
Then there exists an $\mathbb{E}_{\infty}$-orientation 
$$g_{eisen}:MU[\frac{1}{N},\xi_N] \rightarrow Tmf_1(N)
$$
with characteristic series given by
$$\frac{u}{\exp(u)_{eisen}} = \exp (\sum_{k \geq 2} (G_k + G_{k}^{\chi})\frac{u^k}{k!}).
$$
Furthermore if $f:MU[\frac{1}{N},\xi_N] \rightarrow Tmf_1(N) $ is any other $\mathbb{E}_{\infty}$-orientation  with characteristic series $\exp (\sum_{k\geq 2} F_k \frac{u^k}{k!})$ then
$$F_k \equiv (G_k +G_k^{\chi}) \text{ mod }\: \overline{MF_*}(\Gamma_1(N),\mathbb{Z}[\frac{1}{N},\xi_N]). $$
\end{restatable}
\begin{remark}
While the constructions of $Tmf_1(N)$ may be achieved over $\mathbb{S}[\frac{1}{N}]$, it turns out  that to construct complex $\mathbb{E}_{\infty}$-orientations it seems unavoidable  to further adjoin $N$-th roots of unity and work over $\mathbb{S}[\frac{1}{N},\xi_N]$ (see Remark \ref{conundrum} in Section \ref{constructingmeasures}).
\end{remark}
\subsection{Outline of the Work} The first three sections of this work contains standard material which is presumably well known to the experts and is included only to make this work more self contained. 
We begin in Section 2 by recalling some basic facts about classical and $p$-adic modular forms with level structures and in particular about the Frobenius operator on $p$-adic modular forms. Here some subtleties arise (see Remark \ref{frobeniuslevelstructures}) which will later turn out to be crucial in our construction of $\mathbb{E}_{\infty}$-orientations.
In Section 3 we give a short overview of the spectra $Tmf_1(N)$ of topological modular forms with level structures and their chromatic localizations. In Section 4 we recall the obstruction theory for $\mathbb{E}_{\infty}$-orientations originally due to May and Sigurdsson and later refined in \cite{ABGHR}.
%and the logarithmic cohomology operations of Rezk (\cite{Rezklogarithm}). Here we provide 
To any $\mathbb{E}_{\infty}$-ring spectrum $R$ one can associate a spectrum of units $gl_1R$ which is the higher algebraic analogue of the group of units of a commutative ring. A complex $\mathbb{E}_{\infty}$-orientation
$$m:MU\rightarrow R $$ turns out to be equivalent to a null homotopy of the composite 
$$\Sigma^{-1}bu \rightarrow gl_1 \mathbb{S} \rightarrow gl_1 R .$$
Here the map $gl_1 \mathbb{S}\rightarrow gl_1 R$ is induced by the unit map of $R$ and the map $\Sigma^{-1}bu \rightarrow gl_1\mathbb{S}$ is the stable $j$-homomorphism. 
Next we recall the logarithmic cohomology operations of Rezk (\cite{Rezklogarithm} given by the equivalences 
$$ L_{K(n)}gl_1 R \cong L_{K(n)}R $$ and provide a small correction to a result of \cite{AHR} (see Proposition \ref{k2hecke}) which is of no consequence to the results of Ando, Hopkins and Rezk and Wilson, but turns out to be crucial in this work.
In Section 5 we then turn to the case of $\mathbb{E}_{\infty}$-orientations of the spectra $Tmf_1(N)$. For the first part we very much follow the reasoning of \cite{AHR} and \cite{Dylan} apart from an additional argument required at the prime $p=2$ (see Section \ref{OrientsKtheory}). Instead of looking at null homotopies 
$$\Sigma^{-1}bu \rightarrow gl_1\mathbb{S} \rightarrow gl_1 Tmf_1(N) $$ 
we first (Section 5.1 - 5.3) study null homotopies of the composite 
$$ \Sigma^{-1}bu \rightarrow gl_1\mathbb{S}\rightarrow gl_1 Tmf_1(N) \rightarrow L_{K(1)\vee K(2)}gl_1Tmf_1(N) $$ or, equivalently, maps making the following diagram commute:
$$\xymatrix{\Sigma^{-1}bu \ar[r] & gl_1 \mathbb{S} \ar[d] \ar[r] & gl_1\mathbb{S}/\Sigma^{-1}bu \ar@{-->}[d] \\
& gl_1 Tmf_1(N) \ar[r] & L_{K(1)\vee K(2)} gl_1 Tmf_1(N) }$$
The space of these null homotopies becomes accessible through the aforementioned logarithmic operations and we give a description of its path components in terms of $p$-adic measures (Theorem \ref{orienttmfgamma}). Now while the spectra $Tmf_1(N)$ are $E(2)$-local, their spectra of units are not. However it turns out that the discrepancy is controlable. More precisely, one can show that if $R$ is an $E(n)$-local $\mathbb{E}_{\infty}$-ring spectrum then the homotopy groups $\pi_k D$ of the homotopy fiber 
$$D \rightarrow gl_1 R \rightarrow L_{K(1)\vee \ldots \vee K(n)}gl_1 R $$ are zero for $k>n+1$. In Section 5.4 we then go on to study which maps $ f: gl_1\mathbb{S}/\Sigma^{-1}bu \rightarrow L_{K(1)\vee K(2)}gl_1Tmf_1(N)$ in the diagram above already factor through $gl_1Tmf_1(N)$. This is equivalent to the induced map $bu \rightarrow \Sigma D$ in the diagram
$$ \xymatrix{gl_1 \mathbb{S} \ar[r] \ar[d] & gl_1\mathbb{S}/\Sigma^{-1}bu \ar[d]^{f} \ar[r]& bu \ar@{-->}[d] \\
gl_1 Tmf_1(N) \ar[r] & L_{K(1)\vee K(2)}gl_1Tmf_1(N) \ar[r]& \Sigma D  }$$ being null. At this point we depart from previous approaches to strict commutative orientations for topological modular forms. We recall the previously mentioned obstruction theory of Hopkins and Lawson. They construct a sequence $$x_1 \rightarrow  x_2 \rightarrow \ldots$$ of connective spectra with 
$$ \text{hocolim  }x_i \simeq bu$$ and associated Thom spectra
$MX_i$ with the following properties: \begin{enumerate}
    \item $\pi_0Map_{\mathbb{E}_{\infty}}(MX_1,R) \cong \{ \text{homotopy commutative complex orientations of } R \}$
    \item $L_{E(n)}MX_{p^n} \simeq L_{E(n)}MU$.
\end{enumerate}
Their approach now consists of taking an ordinary, homotopy commutative, complex orientation, i.e. an element in $\pi_0 Map_{\mathbb{E}_{\infty}}(MX_1,R)$, and studying when such a map succesively lifts to $MX_{p^k}$. 
We basically turn this idea upside down and start with an element in $[bu,\Sigma D]$ and show, using known results about the coconnectivity of $\Sigma D$ and the connectivity of the fibers of the maps $x_i \rightarrow x_{i+1}$, that it is equivalent to the zero element exactly when the induced element in $[x_1, \Sigma D]$ is. This question in turn is rather easy to decide.
\begin{restatable}{theo}{criterion}\label{orienttmfgammaall}
There is an injection 
\begin{align*}
    \pi_0 Map_{\mathbb{E}_{\infty}}(MU,Tmf_1(N)\widehat{_p})  \hookrightarrow & [bu, L_{K(1)}\Tmf_1(N)\otimes \mathbb{Q}] \\  \cong &\prod_{k\geq 1} MF_{p,k}\otimes \mathbb{Q} \\
f  \mapsto & \pi_{2k}f
\end{align*}
with image the set of sequences $\{ g_k \}$ of $p$-adic modular forms such that \begin{enumerate}
    \item for all $c\in \mathbb{Z}_p^{\times}$ the sequence $\{(1-c^k)(1-p^{k-1}\text{Frob})g_k \}$ satisfies the generalized Kummer congruences;
    \item for all $c\in \mathbb{Z}_p^{\times}$ $\lim_{k \to \infty}(1-c^k)(1-p^{k-1}\text{Frob})g_k = \frac{1}{p}\log(c^{p-1})$; 
    \item The elements $(1-p^{k-1}\text{Frob})g_k$ are in the kernel of the operator $1-\langle p \rangle U$.
    \item for all $k$ the $g_k \in \overline{MF}_k(\Gamma_1(N)) \subset MF_{p,k}(N,\mathbb{W}) $.
\end{enumerate}
\end{restatable}

Finally in Section 6 we construct measures with the required properties by combining results of Katz (\cite{katzeisenstein}) with the sigma orientation of Ando, Hopkins and Rezk, hence completing the proof of Theorem \ref{msuorientations}.
\subsection{Notations and Conventions}
We will use that term $\mathbb{E}_{\infty}$-ring spectrum synonymously to commutative monoid in any one of the modern category of spectra with strict symmetric monoidal smash product.
The $1$-connective cover of the connective complex K-theory spectrum $ku$ will be denoted by $bu \simeq ku\langle 2 \rangle$. 
All rings are assumed to be unital and commutative.
A $p$-adic ring for us means a ring complete and separated in the $p$-adic topology.
\subsection{Acknowledgements}
I would like to thank Dylan Wilson, Lennart Meier, Tobias Barthel, Martin Olbermann, Andrew Senger, Jan Holz, Viktoriya Ozornova and Bj\"orn Schuster for many helpful discussions.
The author would like to give special thanks to his thesis advisor, Gerd Laures, for his  constant support throughout the time leading to this work.

The project leading to this work was partially funded by the DFG under the "Projekt 322563960: Twisted String Structures".

\section{Modular Forms}
We begin by recalling some basic results concerning elliptic curves and modular forms. All the material in this section is well known to the experts and we ask the reader to consult the references cited for a more detailed exposition.
\subsection{Elliptic curves and Level structures}
\subsubsection{Generalities}
For us an elliptic curve $E/S$ over a base scheme $S$ is a proper smooth morphism $p: E \rightarrow S$ such that the fibres are connected curves of genus one together with a section $e: S \rightarrow E$. We will denote by $\omega_{E/S}$ the invertible sheaf $p_*(\Omega^1_{E/S})$ on $S$.
While not immediately clear from the definition it turns out that an elliptic curve $E/S$  is a (one-dimensional) abelian group scheme (see \cite{katzmazur}, 2.1).

One can impose further structure on elliptic curves:
Let $E/R$ be an elliptic curve over a ring $R$ and denote with $E[N]$ for each integer $N\geq 1$ the scheme theoretic kernel of multiplication by $N$. This is a finite flat commutative group scheme over $R$ of rank $N^2$, \'etale if and only if $N$ is invertible in $R$ (see \cite{katzmazur} Corollary 2.3.2). 
\begin{defi}
A $\Gamma_1(N)$-level structure on an elliptic curve $E/R$ over a ring $R$ in which $N$ is invertible is an inclusion of subgroup schemes
$$\beta_N: \mathbb{Z}/N\mathbb{Z} \hookrightarrow E[N].
$$
\end{defi}
Note that the identiy section provides trivially on every elliptic curve a $\Gamma_1(1)$-level structure.
\begin{remark}
One can define more generally level structures on elliptic curves over arbitrary schemes $S$, not just affine ones as we did (see Chapter 3 of \cite{katzmazur}). However all the cases we will eventually be interested will fit the definition above.
\end{remark}
\begin{remark}\label{arithmeticvsnaive}
Sometimes in the literature (for example in \cite{gouvea}) another notion of $\Gamma_1(N)$-level structure is considered: An arithmetic $\Gamma_1(N)$-level structure on an elliptic curve $E/R$ is an inclusion of subgroup schemes
$$\beta_{N}^{arith}: \mu_N \hookrightarrow E[N].
$$
In this context $\Gamma_1(N)$-level structures as we defined them above are called "naive" to distinguish them from the "arithmetic" ones.
If the ground ring $R$ contains an $N$-th root of unity $\xi_N$ both notions agree (the group schemes $\mu_n$ and $\mathbb{Z}/N\mathbb{Z}$ are isomorphic in this case) but in general they specify different structures. 
It can however be shown that one can uniquely (up to isomorphism) associate to an elliptic curve $E$ with a naive $\Gamma_1(N)$-level structure one with an arithmetic level structure by passing to the quotient curve $E/\beta_N(\mathbb{Z}/N\mathbb{Z})$. Conversely if $E^{\prime}$ is an elliptic curve with an arithmetic $\Gamma_1(N)$-level structure, the quotient $E^{\prime}/\beta^{arith}_N(\mu_N)$ can be equipped with a naive $\Gamma_1(N)$-level structure. For more on this we refer the reader to Chapter II of \cite{katzeisenstein2} or the appendix of \cite{viktoriyalennart}.
\end{remark}

\subsubsection{Representability}\label{modularforms}
The moduli problem "elliptic curves over $\mathbb{Z}$ and isomorphisms between such" is represented by a Deligne-Mumford stack denoted by $\mathcal{M}_{ell}$ (see for example \cite{conrad} and the references therein). It comes equipped with a line bundle $\omega$ whose fibre over an elliptic curve $E/R$ are $\omega_{E/R}$. We can base change from $\mathbb{Z}$ to other rings $R$ and regard only elliptic curves and their isomorphisms over $R$-schemes. In this case we will denote the corresponding moduli stacks with $\mathcal{M}_{ell}(R)$. 
Similarily there are moduli stacks $\mathcal{M}_1(N)$ classifying elliptic curves with $\Gamma_1(N)$-level structures together with a line bundle $\omega_{\Gamma}$ induced from $\omega$ by pulling back along the forgetful map of stacks
$$\mathcal{M}_1(N; \mathbb{Z}[\frac{1}{N}]) \rightarrow \mathcal{M}_{ell}(\mathbb{Z}[\frac{1}{N}]). 
$$

For $N\geq 4$ the moduli stacks $\mathcal{M}_1(N)$ are actually isomorphic to smooth affine curves over $\mathbb{Z}[\frac{1}{N}]$ which are usually denoted $Y_1(N)$ (see \cite{katzmazur} 2.7.3 and 4.7.0).
\begin{remark}
While for $N=2,3$ the moduli problems "isomorphism classes of elliptic curves over $\mathbb{Z}[\frac{1}{N}]$ with a $\Gamma_1(N)$-level structure" are not represented by affine schemes anymore, they are still representable through \emph{graded} affine schemes (see \cite{hillmeier} Proposition 4.5).
\end{remark}
The stack $\mathcal{M}_{ell}$ admits a compactification $\overline{\mathcal{M}_{ell}}$ by "glueing in" a copy of the multiplicative group $\mathbb{G}_m$. This stack classifies so called \emph{generalized elliptic curves}: this means we do not insist that the fibres of $p:E \rightarrow S$ are neccessarily smooth curves of genus one anymore, but instead also allow curves with nodal singularities away from the point marked by the section $e$ (see \cite{dr} for the precise definition and the construction of the corresponding moduli stacks).
Similarily we can construct compactifications of $\mathcal{M}_1(N;\mathbb{Z}[\frac{1}{N}])$ 
where we allow the fibers over geometric points also to degenerate to so called N\'eron $m$-gons for $m|N$: These are obtained by glueing $m$-copies of $\mathbb{P}^1$ by identifying the point $\infty$ of the $i$-th copy with the point $0$ of the $i+1$-th copy and the point $\infty$ of the $m$-th copy of $\mathbb{P}^1$ with the point $0$ of the first. Note that a N\'eron $1$-gon is simply a copy of the multiplicative group as we have $\mathbb{P}^1 / (0 \sim \infty)  \cong \mathbb{G}_m$ (see \cite{dr}).
 
For $N>3$ the moduli stacks $\overline{\mathcal{M}_1(N)}$ are isomorphic to proper smooth projective curves. In the cases $N=2,3$ the moduli stacks $\overline{\mathcal{M}_1(N)}$ are weighted projective lines (\cite{hillmeier} Proposition 4.5).
The scheme $\overline{\mathcal{M}_1(N)} - \mathcal{M}_1(N) $ is finite and \'etale over $\mathbb{Z}[\frac{1}{N}]$ and over $\mathbb{Z}[\frac{1}{N},\xi_N]$ it is a disjoint union of $\varphi(N)$ sections called the \textbf{cusps} of $\overline{\mathcal{M}_1(N)}$ (here $\varphi$ denotes Euler's totient function).

\subsection{Classical Modular Forms}
We now turn to modular forms, for the most part we shall follow the treatment in \cite{Katzpadicproperties} where complete proofs of the statements in this section can be found.
\subsubsection{Generalities}
\begin{defi}
The ring of holomorphic modular forms of level $\Gamma_1(N)$ over a ring $R$ is defined as
$$\overline{MF}_*(N,R) \overset{\text{def}}{=} \bigoplus_{k} H^0(\overline{\mathcal{M}_1(N;R)}, \omega^{\otimes k}). 
$$
Similarily meromorphic modular forms of level $\Gamma_1(N)$ are defined as global sections over $\mathcal{M}_1(N)$
\end{defi}
One can spell this out in more concrete terms:
A modular form $f\in MF_{k}(\Gamma, R)$ of weight $k$ is a rule which assigns to every triple $(E/B,\omega,\beta)$ consisiting of an elliptic curve $E$ over an $R$-algebra $B$, an invariant differential $\omega$ on $E$ and a level $\Gamma$-structure $\beta$ on $E$ an element $f(E/B,\omega,\beta)\in B$ such that
\begin{enumerate}
    \item the value $f(E/B,\omega, \beta)\in B$ depends only on the isomorphism class of the triple $(E/B,\omega,\beta)$;
    \item its formation commutes with base change;
    \item it is homogenous of degree $-k$ in the second variable: for every $\lambda \in B^{\times}$
    $$f(E/B,\lambda\cdot \omega, \beta) = \lambda^{-k} f(E/B, \omega, \beta).
    $$
\end{enumerate}

Note that in general it is not true that $$\overline{MF}_k(\Gamma,\mathbb{Z}) \otimes R \cong \overline{MF}_k(\Gamma,R). $$
The reason is that there might exist modular forms over $\mathbb{F}_p$ which do not arise as the mod $p$ reduction of modular forms defined over $\mathbb{Z}$.

We have however the following (\cite{Katzpadicproperties}, 1.7):
\begin{prop}
Let $N\geq 3$ and either $k\geq 2$ or $ N\leq 11$.
Then for any $\mathbb{Z}[\frac{1}{N}]$-algebra $R$  the canonical map 
$$ \overline{MF}_k(\Gamma_1(N),\mathbb{Z}[\frac{1}{N}])\otimes R \rightarrow \overline{MF}_k(\Gamma_1(N),R)
$$
is an isomorphism.
\end{prop}

\subsubsection{The $q$-Expansion Principle}\label{qexpansion}
A prominent r\^ole in the theory of modular forms is played by the Tate curve $\text{Tate}(q)$ which is a generalized elliptic curve over $\mathbb{Z}[[q]]$ defined as 
$$ \text{Tate}(q) \overset{\text{def}}{=} \mathbb{G}_m / q^{\mathbb{Z}}.$$
It corresponds to the formal neighborhood of the cusps of the compactified moduli stack $\overline{\mathcal{M}_{ell}}$. The Tate curve comes equipped with a canonical invariant differential $\omega_{can}$ induced from the canonical invariant differential of $\mathbb{G}_m$. We will denote with $\text{Tate}(q)_R$ the Tate curve defined over the ring $\mathbb{Z}((q))\otimes_{\mathbb{Z}}R$ (see \cite{dr} Section VII and \cite{hilllawson} 3.4).

Note that over the ring $\mathbb{Z}[\frac{1}{N},\xi_N]((q))$ the Tate curve admits a canonical \footnote{after choosing a primitive $N$-th root of unity $\xi_N$, which is not canonical} $\Gamma_1(N)$-level structure given by setting
\begin{align*}
    \beta_{N}:\mathbb{Z}/N\mathbb{Z} &\hookrightarrow \text{Tate}(q)[N] \\
n &\mapsto \xi_N^n.
\end{align*}
In the case of the moduli schemes $\overline{\mathcal{M}_1(N)}$ for $N \geq 3$ the set of the $\varphi(N)$ cusps is (after base change to $\mathbb{Z}[\frac{1}{N},\xi_N]$) isomorphic to the set of the $\varphi(N)$ distinct level $\Gamma_1(N)$-level structures on the Tate curve given by 
\begin{align*}
    \beta: \mathbb{Z}/N\mathbb{Z} & \hookrightarrow \text{Tate}(q)[N] \\
    n & \mapsto \xi_N^{mn}
\end{align*}
for $m\in \varphi(N)$ (\cite{Katzpadicproperties} 1.4).

\begin{remark}
In the absence of an $N$-th root of unity, the Tate curve $\text{Tate}(q)$ no longer has a $\Gamma_1(N)$-level structure. However the curve 
$$ \text{Tate}(q^N) \overset{\text{def}}{=} \mathbb{G}_m/q^{N\mathbb{Z}}
$$
over $\mathbb{Z}[\frac{1}{N}]((q))$ admits a canonical $\Gamma_1(N)$-level structure given by
\begin{align*}
    \beta_N: \mathbb{Z}/N\mathbb{Z} &\hookrightarrow \text{Tate}(q^N)[N] \\
    n & \mapsto q^n \: \text{ mod}\: q^{N\mathbb{Z}}.
\end{align*}

The Tate curve $\text{Tate}(q)$ over $\mathbb{Z}[\frac{1}{N}]((q))$ does admit a canonical \emph{arithmetic} $\Gamma_1(N)$-level structure
$$ \beta_N^{arith}:\mu_N \hookrightarrow \text{Tate}(q)[N]
$$
given by $$
\beta_N^{arith}(\zeta^n) = \zeta^n \: \text{ mod}\: q^{\mathbb{Z}},$$
i.e., $\beta_N^{arith}$ is induced from the canonical inclusion $i:\mu_N \hookrightarrow \mathbb{G}_m$.
\end{remark}
\begin{defi}
The $q$-expansion of a modular form $f\in MF_*(R)$ is given by evaluating it on (one of) the Tate curve(s) $\text{Tate}_R(q)$:

\begin{align*}
    MF_*(R) &\rightarrow \mathbb{Z}((q))\otimes_{\mathbb{Z}} R\\
    f & \mapsto f(\text{Tate}(q),\omega_{can}).
\end{align*}
\end{defi}
The $q$-expansion detects if a modular form is holomorphic:
\begin{prop}[see \cite{Katzpadicproperties}, Section 1.5]
A modular form $f\in MF_*(R)$ extends to $\overline{\mathcal{M}}_{ell}(R)$ if and only if one (and therefore all) of its $q$-epansions already takes values in the subring $\mathbb{Z}[[q]]\otimes_{\mathbb{Z}}R$.
\end{prop}

In the presence of a $\Gamma_1(N)$-level structure we have the following version of the $q$-expansion principle.
\begin{prop}[\cite{Katzpadicproperties} 1.6]
Let $N\geq 3$, $R$ a $\mathbb{Z}[\frac{1}{N}]$-algebra, $R^{\prime} \subset R$ a $\mathbb{Z}[\frac{1}{N}]$-submodule. Let $f$ be a holomorphic modular form of weight $k$ and level $\Gamma_1(N)$ with coefficients in $R$. \begin{enumerate}
\item Suppose that on each of the $\varphi(N)$ connected components of $\overline{\mathcal{M}_{ell}}(\Gamma_1(N),\mathbb{Z}[\frac{1}{N},\xi_N])$ there is at least one cusp at which the $q$-expansion of $f$ vanishes identically. Then $f=0$.

\item
Suppose that on each of the $\varphi(N)$ connected components of $\overline{\mathcal{M}_{ell}}(\Gamma_1(N),\mathbb{Z}[\frac{1}{N},\xi_N])$ there is at least one cusp at which all the coefficients of the $q$-expansion of $f$ lie in $R^{\prime}\otimes_{\mathbb{Z}[\frac{1}{N}]}\mathbb{Z}[\frac{1}{N},\xi_N]$. Then $f$ is a modular form with coefficients in $R^{\prime}$.
\end{enumerate}
\end{prop}

\subsubsection{Hecke Operators}\label{Heckeoperators}
We will now introduce two operators on the ring of modular forms which will become central to our work later. Our exposition follows \cite{Katzpadicproperties}, Section 1.11.

Let $f\in  MF_*(\Gamma_1(N))$ and let $(E,\beta,\omega)$ consist of an elliptic curve $E$ together with a $\Gamma_1(N)$-level structure $\beta$ and an invariant differential $\omega$. For every $n \in (\mathbb{Z}/N\mathbb{Z})^{\times}$ one can define operators  $$\langle n \rangle: MF_*(\Gamma_1(N)) \rightarrow MF_*(\Gamma_1(N))$$ by the equation
$$
\langle n \rangle f(E,\beta,\omega) = f(E,n\beta,\omega)
$$
where $\mathbb{Z}/N\mathbb{Z}^{\times}$ acts on $\beta$ in the obvious way. These maps $\langle n \rangle$ are called diamond operators.
%Note that a modular form $f\in MF_*(\Gamma_1(N))$ is already contained in $MF_*(\Gamma_0(N))$ if and only if all the diamond operators act trivially.

Let $E/R$ be an elliptic curve together with a $\Gamma_1(N)$-level structure $\beta$ for $N\geq3$  over a ring $R$ such that $N$ is invertible and let $p$ be a prime such that $p \nmid N$. Then there exists a finite \'etale over-ring $R^{\prime}$ of $R[\frac{1}{p}]$ such that the subgroup scheme $E[p]$ of points of order $p$ becomes (non-canonically) isomorphic to $(\mathbb{Z}/p\mathbb{Z})^2_{R^{\prime}}$ (\cite{katzmazur} Theorem 6.6.2). Thus over $R^{\prime}$ the elliptic curve $E/R^{\prime}$ has exactly $p+1$ subgroup schemes of order $p$. Let $H$ be one of them and denote with $$
\pi: E_{R^{\prime}} \rightarrow E_{R^{\prime}}/H$$
the map onto the quotient and by 
$$\check{\pi}: E_{R^{\prime}}/H \rightarrow E_{R^{\prime}}
$$
the dual isogeny. Both maps are finite \'etale of rank $p$.
The composition $\pi\circ \check{\pi}$ induces multiplication with $p$ on $E_{R^{\prime}}/H$ and $\check{\pi}\circ \pi$ similarily on $E_{R^{\prime}}$.
If $\omega$ is an invariant differential on $E/R$, then $\check{\pi}^*(\omega_{R^{\prime}})=\text{trace}_{\pi}(\omega_{R^{\prime}})$ is an invariant differential on $E_{R^{\prime}}/H$.  
Observe that since the degree $p$ of the isogeny $\pi$ is prime to $N$, $\pi$ induces an isomorphism of group schemes
$$\pi: E[N] \xrightarrow{\cong} E/H[N]
$$
and we can define a $\Gamma_1(N)$-level structure $\beta^{\prime}$ on $E_{R^{\prime}}/H$ 
by setting 
$$ \beta^{\prime}=\pi(\beta).
$$
The Hecke operators $T_p$ for $p$ prime to $N$ for a modular form $f\in MF_k(\Gamma,R)$ are now defined by 
$$
(T_p)f(E,\omega,\beta) =p^k\sum f(E/H,\check{\pi}^*(\omega),\pi(\beta)) 
$$
where the sum runs over the $p+1$ subgroups of order $p$ of $E/R^{\prime}$. Note that while a priori the Hecke operator takes values in $MF(\Gamma_1(N),R^{\prime})$ an easy calculation with the Tate curve at the standard cusp shows that the effect of $T_p$ on the $q$-expansion $f(q)=\sum_{n\geq 0} a_n q^n$ of a modular form $f\in MF_k(\Gamma_1(N),R)$
is given by
$$T_p f(q) = p^{k-1}\sum_{n \geq 0} \langle p \rangle a_n q^{pn} + \sum_{n\geq 0} a_{np}q^n 
$$
where $\langle p \rangle a_n$ are the coefficients in $q$-expansion of $\langle p \rangle f$ and we see that all the coefficients are already in $R$. Hence by the $q$-expansion principle we can deduce that the operator $T_p$ defines an endomorphism of $MF_*(\Gamma_1(N),R)$.

\subsubsection{The Eisenstein Series}\label{eisensteinseries}
Let $\chi$ be a Dirichlet character modulo $N\geq 3$ such that \mbox{$\chi(-1)=-1^k$}. For simplicity we assume that $N$ is prime (for the general statements and proofs of this section see for example \cite{diamondshurman} Section 4). Then there exist modular forms $G_k^{\chi}$ of weight $k$ and level $\Gamma_1(N)$ over $\mathbb{Z}[\frac{1}{N},\xi_N]$ with $q$-expansion at the standard cusp given by
$$ G_k^{\chi} (\text{Tate}(q),\omega_{can},\xi)= L(1-k,\chi) + 2\sum_{n\geq 1}q^n \sum_{d|n}d^{k-1}\chi(d)
$$
where  $L(1-k,\chi)$ denotes the value at $s= 1-k$ of the $L$-series
$$L(s,\chi) \overset{\text{def}}{=}\sum_{n\geq 1}\chi(n)\cdot n^{-s}.
$$
The action of the diamond operators $\langle p \rangle$ on the $q$-expansion is given by
$$
(\langle p \rangle \; G_k^{\chi}) (\text{Tate}(q),\omega_{can},\xi)= \langle p \rangle L(1-k,\chi) + 2\sum_{n\geq 1}q^n \sum_{d|n}d^{k-1}\chi(pd)
$$
where $\langle p \rangle L(1-k, \chi) = \sum_{n\geq 1} \chi(np)\cdot n^{k-1} $.
For $\chi=1$, the trivial character, the classical Eisenstein series $G_k$ of level 1 are defined completely analogously. Note however that that for the weights $k=1$ and $2$ these are not modular forms anymore.

The importance of the Eisenstein series reveals itself once one applies the Hecke operators $T_{p}$ from above: An easy calculation with $q$-expansions (see for example \cite{diamondshurman} Proposition 5.2.3) gives

$$(T_{p} \:G_k^{\chi}) (\text{Tate}(q),\omega_{can},\xi)= (1- \langle p \rangle p^{k-1})G_k^{\chi}.
$$
It follows that the Eisenstein series are eigenforms for the Hecke operators. In fact they are the only non-cuspidal eigenforms, that is the only eigenforms with non-vanishing constant coefficients in their $q$-expansions.

\subsection{$p$-Adic Modular Forms}
Let $\overline{\mathcal{M}_1(N)}\widehat{_p}$ denote the $p$-completion of $\overline{\mathcal{M}_1(N)}$. This is a formal Deligne-Mumford stack, which means that for a $p$-adic ring $R$ its $R$-points are defined by
$$\overline{\mathcal{M}_1(N)}(R) = \varprojlim \overline{\mathcal{M}_1(N)}(R/p^i)
.$$
Let $\mathcal{M}_1(N)^{ord}$ denote the substack of $\overline{\mathcal{M}_1(N)}\widehat{_p}$ classifying generalized elliptic curves with $\Gamma_1(N)$-level structure with ordinary reduction over $p$-adic rings. 

\begin{defi}
Let $R$ be a $p$-adic ring. The space of holomorphic $p$-adic modular forms $MF_{p,*}(N,R)$  over $R$ with level $\Gamma_1(N)$ is defined as the ring of global sections

$$ 
MF_{p,k}(N,R) = H^0(\mathcal{M}_1(N)^{ord}\widehat{\otimes}R,\omega^{\otimes k}).
$$
\end{defi}
\begin{remark}
This corresponds to $p$-adic modular forms in the sense of \cite{Katzpadicproperties} with "growth condition $r=1$".
\end{remark}
In more concrete terms, a $p$-adic modular form of weight $k$ and level $\Gamma_1(N)$ $f\in MF_{p,k}(N,R)$ is a rule which assigns to each triple $(E/S,\beta,\omega)$ consisting of an elliptic curve $E$ over a $p$-adically complete and separated $R$-algebra $S$ with ordinary reduction, a $\Gamma_1(N)$-level structure 
$$
\beta: \mathbb{Z}/N\mathbb{Z} \hookrightarrow E[N]
$$
and an invariant differential $\omega$ on $E$ a value in  $f(E/S,\beta,\omega)\in S$ subject to the following rules:
\begin{enumerate}
    \item The value $f(E/S,\beta,\omega)$ only depends on the isomorphism class of the triple.
    \item If $\lambda \in S^{\times}$ then $f(E/S,\beta, \lambda \cdot \omega)=\lambda^{-k}f(E/S, \beta, \omega)$.
    \item The formation of $f(E/S,\beta,\omega)$ commutes with base change.
\end{enumerate}

It is clear that every classical holomorphic modular form $f\in \overline{MF_*}(N),R)$ over a $p$-adic ring $R$ gives rise to a $p$-adic modular form.

\subsection{Generalized Modular Forms and the Frobenius Endomorphism}\label{padicmodularfunctions}

\begin{defi}
Let $E$ be an elliptic curve over a $p$-adic ring $B$. A trivialization of $E/B$ is an isomorphism
$$ \varphi: \widehat{E} \xrightarrow{\cong} \mathbb{G}_m
$$
from the formal group $\widehat{E}$ associated to $E$ and the multiplicative formal group $\widehat{\mathbb{G}_m}$.
\end{defi}
\begin{remark}A trivialized elliptic curve has necessarily ordinary reduction. Conversely, given an elliptic curve with ordinary reduction at all primes one can find a trivialization after a base change (see \cite{katz75a}). 
\end{remark}
%We call an $\Gamma_1(Np^{\nu})^{arith}$-level structure 
%$$\beta:\mu_{Np^{\nu}} \hookrightarrow E[Np^{\nu}] 
%$$
%compatible with a trivialization $\varphi$ if the induced map
%$$\mu_{Np^{\nu}}\hookrightarrow \widehat{E} \xrightarrow{\cong} \widehat{\mathbb{G}_m} $$ is the canonical inclusion.
For $p \nmid N$ we can define (naive) $\Gamma_1(N)$-level structures on a trivialized elliptic curve $E/B$ as before as an inclusion
$$ \beta_N: \mathbb{Z}/N\mathbb{Z} \hookrightarrow E[N].
$$
For the case $p|N$ more care is needed (see \cite{katzeisenstein2} 5.0.5) but we do not need this level of generality and refer the interested reader to loc.cit. 
%\begin{remark}
%Notice that a trivialization gives rise to a sequence of mutually compatible $\Gamma_1(p^{\nu})^{arith}$-level structures for all $\nu$. In fact giving a trivialization is equivalent to giving such a sequence of level structures (this is often called a $\Gamma_1(p^{\infty})^{arith}$-level structure). This is basically a consequence of the Serre-Tate theorem (see \cite{gouvea} p.4).
%\end{remark}
\begin{prop}[\cite{katzeisenstein2} 5.1.0]
The functor from the category of $p$-adic rings (with continous homomorphisms) to sets which assigns to every p-adic ring $A$ the set of isomorphism classes of triples $(E,\varphi,\beta)$ consisting of \begin{itemize}
    \item $E$ an elliptic curve over $A$;
    \item A trivialization $\varphi: \widehat{E}\xrightarrow{\cong} \widehat{\mathbb{G}_m}$  of $E$;
    \item A $\Gamma_1(N)$-level structure $\beta:\mathbb{Z}/N\mathbb{Z} \hookrightarrow E[N]$  
\end{itemize}
is representable by a $p$-adic ring $W(\Gamma_1(N),\mathbb{Z}_p)$. For  any $p$-adic ring $B$ the restriction to $B$-algebras is represented by $W(\Gamma_1(N),B) = W(\Gamma_1(N),\mathbb{Z}_p)\hat{\otimes}_{\mathbb{Z}_p} B$. 
\end{prop}
An element $f\in W(\Gamma_1(N),\mathbb{Z}_p)$ is called a generalized modular function.
Given a trivialized elliptic curve $(E/B, \varphi, \beta)$ over a $p$-adic ring $B$ with level structure and a generalized $p$-adic modular function $f\in W(\Gamma_1(N),\mathbb{Z}_p)$ we get a value 
$$f(E/B,\varphi, \beta) \in B
$$
which depends only on the isomorphism class of the triple. Furthermore this process commutes with base change of $p$-adic rings. 
Let $\text{Tate}(q^N)/\widehat{\mathbb{Z}_p((q))}$ be the Tate curve over the $p$-adic completion of the ring $\mathbb{Z}_p((q))$. As the Tate curve is defined as the quotient $$
\mathbb{G}_m/q^{N\mathbb{Z}}
$$
it follows immediately that it comes equipped with a canonical trivialization
$$ \varphi_{can}: \widehat{\text{Tate}(q^N)} \rightarrow \widehat{\mathbb{G}_m}
$$
(its formal group "is" $\widehat{\mathbb{G}_m}$). 
The evaluation of a generalized modular function $f\in W(\Gamma_1(N),B)$ on the Tate curve gives a map 
\begin{align*}
    W(\Gamma_1(N)) &\rightarrow \widehat{B((q))}\\
    f & \mapsto f(Tate(q^N)/\widehat{B((q))}
\end{align*}
which, again, is called the $q$-expansion of $f$. 
We call $f\in W(\Gamma_1(N),B)$ holomorphic if its $q$-expansion already takes value in $\widehat{B[[q]]}$. Denote with $V(\Gamma_1(N),B)$ the subring of $W(\Gamma_1(N),B)$ consisting of holomorphic generalized modular functions.

We have the following theorem which tells us that the $q$-expansion for generalized $p$-adic modular forms plays a similar role to the $q$-expansion in the classical setting:

\begin{theo}[see \cite{katzeisenstein2} 5.2.3]\label{padicqexpansion}
Let $B \subset B^{\prime}$ be p-adic rings. \begin{enumerate}
    \item 
    The $q$-expansion map
    $$W(\Gamma_1(N),B) \rightarrow \widehat{B((q))}
    $$
    is injective and its cokernel is flat over $B$.
    \item We have a natural inclusion $W(\Gamma_1(N),B) \hookrightarrow W(\Gamma_1(N),B^{\prime})$ such that an element $f\in W(\Gamma_1(N),B^{\prime})$ is already an element in $W(\Gamma_1(N),B)$ if and only if its $q$-expansion $f(q)$ is already contained in $\widehat{B((q))}$.
\end{enumerate}
\end{theo}

\subsubsection{Diamond Operators} 
Let $G(N)=\mathbb{Z}_p^{\times}\times (\mathbb{Z}/N)^{\times}.$ Given an element $(x,y)\in G(N)$ we can define the following operator on $V(\Gamma_1(N),B)$ (also called diamond operator) by the formula
$$
\langle x,y \rangle f(E,\varphi,\beta) = f(E,x^{-1}\varphi, y \beta)
$$
where the action of $y$ on $\beta$ is the usual one and the action of $x^{-1}$ on $\varphi$ comes via automorphisms of $\widehat{\mathbb{G}}_m$. The diamond operators are ring homomorphisms. With the help of these operators we can define weight and nebentypus of a generalized modular function:
Let $\chi :\mathbb{Z}_p^{\times}\rightarrow B^{\times}$ be a continous character. We say a generalized modular function $f\in V(\Gamma_1(N),B)$ has weight $\chi$ as modular function of level $N$
if $$\langle x,1 \rangle f = \chi(x) f
$$
for all $x \in \mathbb{Z}_p^{\times}$. In the special case $\chi(x) = x^k$ for some $k\in \mathbb{Z}$ we say $f$ has weight $k$. Finally if $k\in \mathbb{Z}$ and $\epsilon$ is a character of $\mathbb{Z}/N$ we say $f$ has weight $k$ and nebentypus $\epsilon$ if 
$$\langle x,y \rangle f = x^k \epsilon(y)f .
$$

We will call those generalized $p$-adic modular functions which posess a weight $\chi$ generalized $p$-adic modular forms (most generalized $p$-adic modular functions do not posess a weight!). 

\subsubsection{The Relationship to Classical Modular Forms}
We would now like to justify the name "generalized modular forms" by relating these to the previously defined classical and $p$-adic modular forms. 
To do so let $E/B$ be an elliptic curve with trivialization 
$$\varphi:\widehat{E} \xrightarrow{\cong} \widehat{\mathbb{G}_m}.
$$ Take the standard invariant differential $\frac{dT}{1+T}$ on $\widehat{\mathbb{G}_m}$ and pull it back along the trivialization to obtain a nonvanishing invariant differential $\varphi^*(\frac{dT}{1+T})$ on $\widehat{E}$ which is the restriction to $\widehat{E}$ of a unique invariant differential $\omega_{\varphi}$ on $E$. If $B$ is flat over $\mathbb{Z}_{p}$ the trivialization $\varphi$ is determined by $\omega_{\varphi}$ (\cite{katzeisenstein2} 5.4.0). 
\begin{prop}[see \cite{katzeisenstein2} 5.4]
The construction $(E/B,\varphi,\beta) \rightarrow (E/B, \varphi^*(\frac{dT}{1+T}),\beta)$ defines by transposition for each $k$ an injective morphism of $B$-modules
\begin{align*}
    \overline{MF}_k(N,B) &\rightarrow V(\Gamma_1(N),B)\\
f & \mapsto \widetilde{f}
\end{align*}
and hence a morphism of rings 
$$ \overline{MF}_*(\Gamma_1(N),B)= \bigoplus_{k}\overline{MF_k}(\Gamma_1(N),B) \rightarrow V(\Gamma_1(N),B)
$$
given by 
$$
\widetilde{f}(E,\varphi,\beta) \overset{\text{def}}{=} f(E,\varphi^*(\frac{dT}{1+T}),\beta).
$$
This ring homomorphism is injective if $B$ is flat over $\mathbb{Z}_p$. It furthermore preserves $q$-expansions, so the following diagram commutes
$$\xymatrix{ \overline{MF}_*(\Gamma_1(N),B) \ar[d] \ar[r] & V(\Gamma_1(N),B) \ar@{^{(}->}[d] \\ 
B((q)) \ar@{^{(}->}[r] & \widehat{B((q))}
}.
$$
This morphism of rings is compatible with the diamond operators. 
\end{prop}
We have a completely analogous situation for $p$-adic modular forms $MF_{p,*}(\Gamma_1(N),B)$, by the same construction as above we also get an inclusion (again under the assumption that $B$ is flat over $\mathbb{Z}_p$)
$$MF_{p,*}(\Gamma_1(N),B) \hookrightarrow V(\Gamma_1(N),B)).
$$
It is clear that classical and $p$-adic modular forms of weight $k$ map to generalized modular forms of weight $k$ as defined above and furthermore that this ring homomorphism commutes with classical diamond operators. In fact the converse is also true.
\begin{theo}[\cite{katzeisenstein2} 5.10]Let $B$ be a $p$-adic ring flat over $\mathbb{Z}_p$.
Denote with $GV_*(\Gamma_1(N),B)$  the (graded) subring of $V(\Gamma_1(N),B)$ consisting of all generalized modular forms $f$ with integral weight
$$ \langle x,1 \rangle f = x^k $$
for  $k \in \mathbb{Z}$ and for all $x \in \mathbb{Z}_p^{\times}$.
Then the construction above gives an isomorphism of graded rings 
$$ GV_*(\Gamma_1(N),B) \cong MF_{p,*}(\Gamma_1(N),B).
$$
\end{theo}
%\begin{remark}
%Actually even more is true. Let $\chi:\mathbb{Z}_p^{\times} \rightarrow %\mathbb{Z}_p^{\times}$ be a continous character. Then one can show that the %generalized $p$-adic modular functions over $B$ with weight $\chi$ are exactly the %p-adic modular forms with weight $\chi$ in the sense of Serre \cite{Serre}. As we %will not need those we refer the interested reader to \cite{Katzcongruences}, Prop. %A.1.6.
%\todo{refernce to serre}.
%\end{remark}

\subsubsection{The Frobenius Endomorphism}\label{atkinsoperator}
Let $(E/B,\varphi,\beta)$ be a trivialized elliptic curve over a $p$-adic ring $B$ with $\Gamma_1(N)$-level structure where $p$ is prime to $N$. 
The canonical subgroup $E_{can}$ of $E$ is the (unique) subgroup scheme which extends the kernel of the multiplication by $p$ map on the associated formal group $\widehat{E}$. Via the trivialization $\varphi$ the subgroup scheme $E_{can}$ can be identified with $\mu_p\subset \widehat{\mathbb{G}}_m$.
Denote with 
$$\pi:E \rightarrow E^{\prime}=E/E_{can}
$$
the projection onto the quotient and 
with $\check{\pi}:E^{\prime}\rightarrow E$ the dual isogeny.
By Cartier-Nishi duality we see that $\ker(\check{\pi})\cong \mathbb{Z}/p$, hence $\check{\pi}$ is \'etale. It follows that $\check{\pi}$ induces an isomorphism between the formal groups of $E$ and $E^{\prime}$ and we can define a trivialization $\varphi^{\prime}$ on $E^{\prime}$ by
$$ \varphi^{\prime}=\varphi \circ \check{\pi}. 
$$
As $p$ is prime to $N$, $\pi$ induces an isomorphism between the kernels of multiplication by $N$ on $E$ and on $E^{\prime}$, so we can define a $\Gamma_1(N)$-level structure $\beta^{\prime}$ on $E^{\prime}$ via 
$$\beta^{\prime}=\pi \circ \beta.
$$

%\begin{remark}\label{alternatelevel}
%Note that we could also define a level structure on $E^{\prime}$ using the dual isogeny %$\check{\pi}$ (which is also of rank $p$) via 
%$$\beta^{\prime}=\check{\pi}^{-1} \circ \beta.
%$$
%This is the convention adopted in \cite{gouvea} whereas our convention coincides with the one %used in \cite{katzeisenstein} and \cite{Katzpadicproperties}.
%\end{remark}
Hence the assignement
$$ (E,\varphi,\beta) \mapsto (E/E_{can},\varphi^{\prime},\beta^{\prime})
$$
defines a functor on trivialized elliptic curves with $\Gamma_1(N)$-level structure
and the Frobenius endomorphism 
$$ \text{Frob}:V(\Gamma_1(N),B) \rightarrow V(\Gamma_1(N),B)
$$
is now defined by the formula
$$
\text{Frob}\;f(E,\varphi,\beta) \overset{\text{def}}{=} f(E/E_{can},\varphi^{\prime},\beta^{\prime}).
$$
It is clear from the definition that the Frobenius endomorphism commutes with the diamond operators and in particular preserves weights. In light of the discussion above it follows immediatley that it also restricts to an endomorphism of the subring $MF_{p,*}(\Gamma_1(N),B)$ of $p$-adic modular forms.
The following lemma shows that the name \lq\lq Frobenius endomorphism" is indeed justified.
\begin{lemma}
For $f\in V(\Gamma_1(N),\mathbb{Z}_p) $ we have
$$\text{Frob}\;f \equiv f^p \: \text{ mod } \: p .$$
\end{lemma}
\begin{proof}
An easy calculation (carried out, for example, in \cite{katzeisenstein} p.10) shows that if $f(q^N)=\sum_0^{\infty}a_n q^n$ is the $q$-expansion of a generalized modular function $f\in V(\Gamma_1(N),\mathbb{Z}_p)$ evaluated on $(\widehat{\text{Tate}(q^N)},\varphi_{can},q)$ then 
$$(\text{Frob}\;f)(q)= \sum_0^{\infty} a_n q^{np}.
$$
Modulo $p$ the Frobenius endomorphism thus reduces to the $p$-th power endomorphism on $q$-expansions and the result follows from the $q$-expansion principle.
\end{proof}
Let $k$ be a perfect field of characteristic $p$ containing an $N$-th root of unity $\xi_0$ ($N\geq 3$). 
Denote with $\mathbb{W}_N$ the ring of Witt vectors of $k$ and denote with $\xi\in \mathbb{W}_N$ the unique lift of $\xi_0$. In the following we will omit the subscript $N$ if it is clear from the context.
For a generalized modular function $f\in V(\Gamma_1(N),\mathbb{W}_N)$ we can just as well evaluate the $q$-expansion on the curve $\widehat{\text{Tate}(q)}$ instead of $\widehat{\text{Tate}(q^N)}$ (we now have the $N$-th root of unity $\xi$ at our disposal and hence - just as in Section \ref{qexpansion} - a canonical $\Gamma_1(N)$-level structure).

\begin{remark}
The analogue of Theorem \ref{padicqexpansion} holds for $p$-adic $\mathbb{W}$-algebras $B$ and the Tate curve $(\widehat{\text{Tate}(q)},\varphi_{can},\xi)$ over $\widehat{B((q))}$. 
\end{remark}
\begin{lemma}
Let $B$ be a $p$-adic $\mathbb{W}$-algebra and let $f(q)=\sum_0^{\infty} a_n q^n$ be the $q$-expansion of a generalized $p$-adic modular function $f\in V(\Gamma_1(N),B)$ evaluated at the curve $(\widehat{\text{Tate}(q)}/\widehat{B((q))},\varphi_{can}, \xi)$.
Then the effect of the Frobenius endomorphism is given by 
$$(\text{Frob}\:f)(q) = \sum_0^{\infty} \langle p \rangle a_n q^{np} 
$$
 where $\langle p \rangle a_n$ denotes the coefficients of the $q$-expansion of $\langle 1,p\rangle f$.
\end{lemma}
\begin{proof}
This is again the same calculation as in \cite{katzeisenstein}, p. 10.
\end{proof}
\begin{coro}
For $f\in V(\Gamma_1(N),\mathbb{W})$ we have 
$$ f \equiv f^p  \: \text{ mod }\:p.
$$
\end{coro}
\begin{proof}
This is the same argument as above after observing that for the coefficients $a_n$ in the $q$-expansion $f(q) =\sum a_nq^n$ of $f$ (when evaluated on $(\widehat{\text{Tate}(q)},\varphi_{can},\xi)$)
$$ a_n \equiv \sum_{i=0}^{N-1} b_i \xi_0^{i} \: \text{ mod } \: p  
$$
for some $b_i \in \mathbb{Z}/p\mathbb{Z}$. 
The effect of the diamond operator $\langle 1, p \rangle$ is given by 
$$ \langle 1, p \rangle f(\widehat{\text{Tate}(q)}, \varphi_{can}, \xi) = f(\widehat{\text{Tate}(q)}, \varphi_{can}, \xi^p)
$$
so we have 
$$ \langle p \rangle a_n \equiv \sum_0^{N-1} b_i \xi_0^{pi}  \equiv (\sum_0^{N-1}b_i\xi_0^i)^p\equiv (a_n)^p\: \text{ mod } \:p. 
$$
\end{proof}

%\begin{remark}
%If we would use the convention as mentioned in Remark \ref{alternatelevel} defining the Frobenius the effect on a $q$-expansion $f(q)$ would simply be 
%$$
%f(q) \mapsto f(q^p)
%$$
%whithout any modifying action by the diamond operator (this can also be easily be read of by Katz forementioned calculation). As the diamond operators are automorphisms of $V(\Gamma_1(N))$ we see that these two possible ways of defining the Frobenius only differ by said automorphism which will allow us to use most of the results of \cite{gouvea} with little or no modification. 
%\end{remark} 
%\begin{prop}
%On $V(\Gamma_1(N),\mathbb{F}_p)$ the Frobenius endomorphism 
%$$\text{Frob}:V(\Gamma_1(N), \mathbb{F}_p) \rightarrow V(\Gamma_1(N),\mathbb{F}_p) $$
%is just the usual $p$-th power endomorphism 
%$$x \mapsto x^p.
%$$
%\end{prop}
%\begin{proof}
\begin{remark}\label{frobeniuslevelstructures}
Before stating the next result we should mention that one can also define generalized $p$-adic modular functions for trivialized elliptic curves $E$ over $p$-adic rings $B$ with arithmetic $\Gamma_1(N)$-level structures (this form of level structure is actually the more common one considered in the literature). In that case one requires that the arithmetic $\Gamma_1(N)$-level structure
$$\beta_N^{arith}: \mu_N \hookrightarrow E[N]
$$
is compatible with the trivialization $\varphi: \widehat{E} \xrightarrow{\cong} \widehat{\mathbb{G}_m}$ in the sense that the induced map
$$ \mu_N \hookrightarrow \widehat{E} \xrightarrow{\varphi} \widehat{\mathbb{G}_m}
$$
is the canonical inclusion. Again the moduli problem "trivialized elliptic curves over $p$-adic rings together with arithmetic $\Gamma_1(N)$-level structure is represented by a $p$-adic ring $W(\Gamma_1(N)^{arith},\mathbb{Z}_p)$. The subring $V(\Gamma_1(N)^{arith},\mathbb{Z}_p)$ of holomorphic generalized $p$-adic modular functions is analogously defined using the Tate curve $\widehat{\text{Tate}(q)}$ with its canonical \emph{arithmetic} $\Gamma_1(N)$-level structure. Under our assumption $p \nmid N$ there is essentially nothing new: By the construction sketched in Remark \ref{arithmeticvsnaive} one can again pass back and forth between arithmetic and naive level structures and one has inverse isomorphisms  (see \cite{katzeisenstein2} 5.6.2)
$$ V(\Gamma_1(N),B) \leftrightarrows V(\Gamma_1(N)^{arith},B)
$$
which preserve $q$-expansions (for the curve ($\widehat{\text{Tate}(q^N)},\varphi_{can},q)$ on the left side and for the curve $(\widehat{\text{Tate}(q)}, \varphi_{can}, \iota)$
on the right side where 
$$ \iota: \mu_N \hookrightarrow \widehat{\text{Tate}(q)} = \mathbb{G}_m/q^{\mathbb{Z}}
$$ 
is the canonical inclusion). In light of this we can freely apply results in the literature such as \cite{gouvea} which are written in terms of arithmetic level structures. There is however a subtle point with regard to the Frobenius endomorphism one should keep in mind: It is again defined via the quotient of trivialized elliptic curve by the canonical subgroup 
$$ \pi:E \rightarrow E/E_{can}
$$
and the induced trivialization on the quotient curve is defined in the same manner as above.
However if one defined the induced $\Gamma_1(N)^{arith}$-level structure $\beta^{\prime}$ on the quotient curve as
$$ \xymatrix{ & E[N] \ar[dd]^{\pi} \\
\mu_N \ar@{^{(}->}[ur]^{\beta} \ar@{^{(}->}[dr]^{\beta^{\prime}} & \\
& E/E_{can}[N]
}
$$
as we did for the naive $\Gamma_1(N)$-level structures one would find that the effect of the operator defined in this manner which we will tentatively call $\widetilde{\text{Frob}}$ on the $q$-expansion $f(q)=\sum a_nq^n$ of a generalized $p$-adic modular function $f\in V(\Gamma_1(N)^{arith},\mathbb{Z}_p)$
were given by $$ \widetilde{\text{Frob}}\; f(q) = \sum \langle p \rangle a_nq^{np}.
$$
In particular it would neither reduce to the $p$-th power map modulo $p$ nor would it be compatible with the aforementioned isomorphisms between the rings of naive and arithmetic generalized $p$-adic modular forms (which preserves $q$-expansions). To remedy this situation one is better advised to define the level structure on the quotient curve using the dual isogeny
$$\check{\pi}:E/E_{can} \rightarrow E.
$$
It is again of degree $p$ prime to $N$ and therefore also induces an isomorphism on the subgroups of $N$-torsion points. One can now define an arithmetic $\Gamma_1(N)$-level structure $\beta^{\prime}$ on $E/E_{can}$ by the formula
$$\beta= \check{\pi}\circ \beta^{\prime}
$$
and define the operator $\text{Frob}^{arith}:V(\Gamma_1(N)^{arith},B) \rightarrow V(\Gamma_1(N)^{arith},B)$ via
$$\text{Frob}^{arith} \:f(E,\varphi,\beta) = f(E/E_{can},\varphi(\check{\pi}),\check{\pi}^{-1}(\beta)).
$$
A straightforward calculation on the $q$-expansion $f(q)=\sum a_n q^n$ of a generalized $p$-adic modular function $f\in V(\Gamma_1(N)^{arith},\mathbb{Z}_p)$ now shows that
$$ \text{Frob}^{arith} f(q) = \sum a_n q^{np}
$$ 
so it is in particular compatible with our definition of the Frobenius endomorphism on generalized $p$-adic modular functions with naive level structure. 
\end{remark}
With this setup in place on may freely apply the following results which are stated in \cite{gouvea} in terms of arithmetic level structures and the operator that we called $\text{Frob}^{arith}$ to our situation of naive level structures.
\begin{theo}[\cite{gouvea} Proposition II2.9]
Let $B$ be a $p$-adically complete discrete valuation ring. Then the Frobenius endomorphism $\text{Frob}:V(\Gamma_1(N),B)\rightarrow V(\Gamma_1(N),B)$ is locally free of rank $p$.
\end{theo}
It follows that we can define  a trace homomorphism
$$\text{tr}_{\text{Frob}}: V(\Gamma_1(N),B) \rightarrow V(\Gamma_1(N),B)
$$
by setting 
$$
(\text{tr}_{\text{Frob}}\:f)(E,\varphi,\beta) = \sum_{(E_i,\varphi_i,\beta_i)\mapsto (E,\varphi,\beta)} f(E_i,\varphi_i,\beta_i)
$$
where the sum is over all trivialized elliptic curves $(E_i,\varphi_i,\beta_i)$ which map under the quotient by the fundamental subgroup to $(E,\varphi,\beta)$.
Now let us again assume that we are working over $\mathbb{W}$ so that we can use the $q$-expansion given by the curve $\widehat{\text{Tate}(q)}$ with its canonical naive $\Gamma_1(N)$-level structure: 
\begin{lemma}[see \cite{Katzpadicproperties} p.22-23 and p.64]
Let $f(q)=\sum a_n q^n$ be the $q$-expansion of a generalized $p$-adic modular function 
$f\in V(\Gamma_1(N),\mathbb{W})$ evaluated at $(\widehat{\text{Tate}(q)},\varphi_{can},\xi)$. Then we have
$$\text{tr}_{\text{Frob}} \: f(q) = p \sum \langle p^{-1} \rangle a_{np}q^{n}.
$$
\end{lemma}
By the $q$-expansion principle the expression $$\frac{1}{p}\text{tr}_{\text{Frob}}$$
is well defined and we can define an operator (often called the Atkin operator, other sources such as \cite{gouvea} call this the Dwork operator) $$U: V(\Gamma_1(N),\mathbb{W} )\rightarrow V(\Gamma_1(N),\mathbb{W})$$  by setting
$$
U(f) \overset{\text{def}}{=} \frac{1}{p}(\text{tr}_{\text{Frob}})f.
$$
It is clear that the $U$-operator commutes with the diamond operators and in particular with weights, hence it restricts to an operator on the ring of $p$-adic modular forms $MF_{p,*}(\Gamma_1(N),\mathbb{W})$. By inspecting the $q$-expansion one immediately sees that
$$
U(\text{Frob} f) = f
$$
and that 
$$ T_p \overset{\text{def}}{=} p^{k-1} \text{Frob} + \langle p  \rangle U
$$
extends the classical Hecke operators of Section \ref{Heckeoperators} to generalized $p$-adic modular functions in $V(\Gamma_1(N),\mathbb{W})$. 
\section{Topological Modular Forms}
\subsection{Generalities}A generalized elliptic cohomology theory consits of a triple $(E,C,\phi)$ where $E$ is a weakly even periodic (and hence complex orientable) spectrum, $C$ is an elliptic curve over $\pi_0E$ and 
$$ \phi: \widehat{C} \xrightarrow{\cong} \widehat{\mathbb{G}}_E 
$$
is an isomorphism between the formal group $\widehat{C}$ associated to the elliptic curve and the formal group of the spectrum $E$.
The map of moduli stacks 
$$\overline{\mathcal{M}_{ell}} \rightarrow \mathcal{M}_{FG} $$ which associates to a generalized elliptic curve its formal group is flat, conversely, the Landweber exact functor theorem implies that every flat map
$$f:\text{Spec}(R) \rightarrow \overline{\mathcal{M}_{ell}} $$ gives rise to a generalized elliptic spectrum. 
Goerss, Hopkins and Miller showed that this construction can be refined to a presheaf of spectra on $\overline{\mathcal{M}_{ell}}$ if one restricts oneself to working in the category of $\mathbb{E}_{\infty}$-ring spectra. More precisely one has the following:
\begin{theo}[\cite{tmfbook}, Chapter 12]
There exists a presheaf $\mathcal{O}^{top}$ of $\mathbb{E}_{\infty}$-ring spectra on the small \'etale site of $\overline{\mathcal{M}_{ell}}$.
\end{theo}
In particular given an affine \'etale open 
$\text{Spec}(R)\xrightarrow{C} \overline{\mathcal{M}_{ell}}$
classifying a generalized elliptic curve $C/R$ the spectrum of sections $E=\mathcal{O}^{top}(\text{Spec}(R))$
is a weakly even periodic $\mathbb{E}_{\infty}$-ring spectrum such that \begin{enumerate}
    \item $\pi_0(E)\cong R$;
    \item The formal group $\widehat{\mathbb{G}_E}$ of $E$ is isomorphic to $\widehat{C}$.
\end{enumerate}
\begin{remark}
Not every elliptic spectrum arises in this fashion (or via the Landweber exact functor theorem). The most prominent examples would be the Morava $K$-theories $K(2)$ of height $2$ corresponding to the formal group of a supersingular elliptic curve.
\end{remark}
The spectrum $Tmf$ is now definied as the spectrum of global sections 
$$
Tmf \overset{\text{def}}{=} \Gamma(\overline{\mathcal{M}_{ell}}, \mathcal{O}^{top}).
$$
The presheaf $\mathcal{O}^{top}$ restricts to a presheaf on $\mathcal{M}_{ell}$ and the resulting spectrum of global sections is denoted $TMF$. 
The maps $\mathcal{M}_1(N) \rightarrow \mathcal{M}_{ell}$ are \'etale after inverting $N$ and so one can define 
$$TMF_1(N)\overset{\text{def}}{=}\Gamma(\mathcal{M}_1(N,\mathbb{Z}[\frac{1}{N}]), \mathcal{O}^{top}).
$$
For the compactified stacks $\overline{\mathcal{M}_1(N)}$ \'etaleness no longer holds over the cusps. However the maps of stacks $\overline{\mathcal{M}_1(N)}\rightarrow \overline{\mathcal{M}_{ell}}$ still possess the weaker property of being log-\'etale. Hill and Lawson constructed in \cite{hilllawson} an extension of $\mathcal{O}^{top}$ to the log-\'etale site of $\overline{\mathcal{M}_{ell}}$ and thus were able to construct spectra $Tmf_1(N)$, again as global sections
$$Tmf_1(N) \overset{\text{def}}{=} \Gamma(\overline{\mathcal{M}_1(N, \mathbb{Z}[\frac{1}{N}])},\mathcal{O}^{top}).
$$
%While for each affine \'etale open 
%$$\text{Spec}(R) \rightarrow \overline{\mathcal{M}_{ell}}
%$$
%the spectrum of sections is complex orientable (in fact Landweber exact)
%this does in general only hold if $\mathcal{M}_{ell}(\Gamma,R)$ is given by an affine curve (with the noteable exceptions %$\Gamma_1(2)$ and $\Gamma_1(3)$ - see \cite{hillmeier}).

The homotopy groups of $Tmf_1(N)$ can be computed via the "elliptic spectral sequence" with $E_2$ term given by the sheaf cohomology
$$E_2^{s,t} = H^s(\overline{\mathcal{M}_1(N)},\omega^{\otimes t/2}) \Rightarrow \pi_{t-s} Tmf_1(N)
$$
where one sets $\omega^{\otimes t/2}=0$ if $t$ is odd (similiar for $TMF_1(N)$). It follows immediately that for $N>2$, as $\mathcal{M}_1(N))$ is isomorphic to an affine curve, that the spectral sequence collapses on the $E_2$ page and $\pi_{2k}TMF_1(N)\cong MF_k(\Gamma_1(N))$ for such $N$. In the cases of the moduli stacks $\overline{M_1(N)}$ for $N>1$  the homotopy groups of $Tmf_1(N)$ are still concentrated in even degrees and torsion free in positive degrees apart from possibly on $\pi_1$ due to Serre-duality (nontrivial torsion only appears for large level $N$). It turns out  that all these spectra are Landweber exact (see \cite{hillmeier} Section 4).
\begin{remark}
On the other hand the homotopy groups of $Tmf$ contains plenty of $2$ and $3$ torsion, both in even and in odd degrees, detecting many elements in the stable homotoy groups of the sphere. The spectra $Tmf$ and $TMF$ only become complex orientable after inverting $6$. For a calculation of $\pi_* Tmf$ see \cite{tmfbook}, Chapter 13.
\end{remark}

\subsection{Chromatic Localizations of Topological Modular Forms}
In this section we will recall well known (to the experts at least) results about the chromatic localizations of $Tmf_1(N)$. All the material in this section can be found or easily deduced from Behren's article on the construction of $Tmf$ (\cite{tmfbook}, Chapter 12),\cite{hilllawson}, and from \cite{Dylan}. 
We will restrict our attention to the level $\Gamma=\Gamma_1(N)$ such that $N\geq 3$ is a prime unless otherwise stated. For these level the theories $Tmf_1(N)$ are complex orientable, see above. Furthermore, to circumvent some technical issues (see Remark \ref{conundrum} in Section \ref{constructingmeasures}) we adopt the following convention.
\begin{convention}For the rest of this paper we will adjoin primitive $N$-th roots of unity everywhere in sight. This means every spectrum is a module respectivley algebra over $\mathbb{S}[\frac{1}{N},\xi_N]$. We will usually drop the $N$-th root of unity from our notation.
\end{convention}
\subsubsection{Rationalization}
We simply have $\pi_{2*}Tmf_1(N) = \overline{MF_*}(\Gamma_1(N),\mathbb{Q}) \cong \overline{MF_*}(\Gamma_1(N))\otimes \mathbb{Q}$.
Note that this holds also for $Tmf_1(1)=Tmf$.
\subsubsection{$K(1)$-local $Tmf_1(N)$} 
Recall from Section \ref{padicmodularfunctions} that the formal affine scheme $\text{Spf}(V(\Gamma_1(N),\mathbb{W})$ represents trivialized elliptic curves with a $\Gamma_1(N)$-level structure over $p$-adic $\mathbb{W}$-algebras. The presheaf $\mathcal{O}^{top}$ pulls back to $\text{Spf}(V(\Gamma_1(N),\mathbb{W})$ and the very construction of topological modular forms now readily implies:

\begin{prop}[see also \cite{Dylan}, Section 2]
We have the following:
\begin{enumerate}
    \item $L_{K(1)}(KU\widehat{_p}\wedge L_{K(1)}Tmf_1(N) \simeq \Gamma(\text{Spf}(V(\Gamma_1(N)),\mathcal{O}^{top});
     $
    \item $L_{K(1)}Tmf_1(N)\simeq \Gamma(\overline{\mathcal{M}_{ell}^{ord}(\Gamma_1(N))}, \mathcal{O}^{top})$
\end{enumerate}
\end{prop}
Note that the ring $V(\Gamma_1(N),\mathbb{W})$ is flat over $\mathbb{Z}_p$ and we get 
$$\pi_{*}L_{K(1)}(KU\widehat{_p}\wedge Tmf_1(N)) \cong KU\widehat{_p}_* \otimes_{\mathbb{Z}_p}V(\Gamma_1(N),\mathbb{W}).$$
As we observed in Section \ref{padicmodularfunctions} the Frobenius operator $\text{Frob}$ on $V(\Gamma_1(N),\mathbb{W})$ reduces modulo $p$ to the $p$-th power map and therefore determines the structure of a $\theta$-algebra on $V(\Gamma_1(N),\mathbb{W})$. We may extend this to a $\theta$-algebra 
on $KU\widehat{_p}\otimes_{\mathbb{Z}_p} V(\Gamma_1(N),\mathbb{W})$ by letting $\text{Frob}$ act on the periodicity generator $u$ of $KU\widehat{_p}_*$ by multiplication with $p$. We hence get a $\theta$-algebra structure on $\pi_{2*}L_{K(1)}Tmf_1(N)$ through \'etale descent as follows.
\begin{prop}\label{theta}
The $\theta$-algebra structure of $L_{K(1)}Tmf_1(N)$ is such that the following diagram commutes 

$$\xymatrix{ \pi_{2k}L_{K(1)}Tmf_1(N) \ar[r] \ar[d]^{\psi} & MF_{k,p}(\Gamma_1(N),\mathbb{W}) \ar[d]^{f \mapsto p^k \text{Frob}f}\\
\pi_{2k}L_{K(1)}Tmf_1(N) \ar[r] & MF_{k,p}(\Gamma_1(N),\mathbb{W}).
}
$$
\end{prop}

\subsubsection{$K(2)$-local $Tmf_1(N)$}\label{k2localtmfsubsect}
Let $\mathcal{M}_1(N)^{ss}$ be the supersingular locus of $\mathcal{M}_1(N)$. This is the formal neighborhood of the finitely many points of $\mathcal{M}_1(N)$ (and therefore also of $\overline{\mathcal{M}_1(N)}$ as the points in a (formal) neighborhood of the cusps have ordinary reduction) corresponding to supersingular elliptic curves with a $\Gamma_1(N)$-level structure. The Serre-Tate theorem now establishes an equivalence between the category of deformations of a supersingular elliptic curve and the deformations of its formal group. The Lubin-Tate rings classifying deformations of a formal group $\widehat{C}$ of height $2$ over a field $k$ of positive characteristic are (non-canonically) isomorphic to $\mathbb{W}(k)[[u]]$. As the $\mathcal{M}_1(N)$ for $N\geq 5$ are affine we have for these level
$$\mathcal{M}_1(N)^{ss} \cong \bigsqcup_{I} \emph{\text{Spf}}(\mathbb{W}(k_i)[[u]])
$$
where the coproduct runs over all isomorphism classes of supersingular elliptic curves $C_i$ with a $\Gamma_1(N)$-level structure over finite fields $k_i$ of characteristic $p$ .
\begin{remark}
One can show that every supersingular elliptic curve is already isomorphic to one defined over either $\mathbb{F}_p$ or $\mathbb{F}_{p^2}$. This cannot hold anymore once we introduce level structures which is easy to see for example in the case of a $\Gamma_1(N)$-level structure i.e. a specified point of order $N$: For an elliptic curve $E$ over $\mathbb{F}_{p^n}$ Hasses's Theorem gives an upper bound for the number of points on $E$ by the formula
$$ ||E/\mathbb{F}_{p^n}| - (p^n +§) | \leq 2 \sqrt{p^n}.
$$ Hence for large level $N$ and small primes $p$ there are simply not enough points on the curve over $\mathbb{F}_{p^2}$ to have a point of exact order $N$. So the number of isomorphism classes of supersingular elliptic curves and the cardinality of the fields $k \subset \overline{\mathbb{F}_p}$ over which they are defined vary with the level structures and the prime.
\end{remark}
All Lubin-Tate rings classifying universal deformations of formal groups admit essentially unique lifts to $\mathbb{E}_{\infty}$-ring spectra, one for each isomorphism class of formal groups of finite height, called the Lubin-Tate spectra. The lift to homotopy commutative ring spectra is due to Morava, the lift to $\mathbb{E}_{\infty}$-ring spectra is due to Goerss, Hopkins and Miller.
In light of this we get for $N\geq5$ as part of the construction of $Tmf$
$$L_{K(2)}Tmf_1(N) \simeq \Gamma(\mathcal{M}_1(N)^{ss},\mathcal{O}^{top})\simeq \prod_{C_i}
E_2(\widehat{C_i})
$$
where $E_2(\widehat{C_i})$ is the Lubin-Tate spectrum corresponding to the universal deformation of the formal group $\widehat{C_i}$ of a supersingular elliptic curve $C_i$ with $\Gamma_1(N)$-level structure and the product is taken over the (finitely many) isomorphisms classes of such curves.

Since the moduli stack $\mathcal{M}_1(3)$ is not affine anymore a slightly modified construction is required (see \cite{lawsonnaumann} Section 4). Let $\mathcal{M}(3)$ be the moduli stack of elliptic curves $C$ with a $\Gamma(3)$-level structure, that is an isomorphism
$$ \alpha:(\mathbb{Z}/3\mathbb{Z})^2 \xrightarrow{\cong} C[3].$$ 
$\mathcal{M}(3)$ is affine and \'etale over $\mathcal{M}_{ell}$ so we have a spectrum $$TMF(3) \overset{\text{def}}{=} \mathcal{O}^{top}(\mathcal{M}(3)).$$ 
and as above
$$ L_{K(2)}TMF(3) \cong \prod_{C_i}
\widetilde{E}_2(\widehat{C_i}) $$
where  $\widetilde{E}_2(\widehat{C_i}$ is the Lubin-Tate spectrum corresponding to the universal deformation of a supersingular elliptic curve with $\Gamma(3)$-level structure and the product runs over all isomorphism classes of such.

The group $GL_2(\mathbb{Z}/3\mathbb{Z})\cong \text{Aut}(\mathbb{Z}/3\mathbb{Z})^2$ acts on an elliptic curve with $\Gamma(3)$-level structure and hence also on $TMF(3)$.

Let $$ G_{\Gamma_1(3)} \overset{\text{def}}{=} \left\{ \begin{pmatrix} 1& * \\
0 & *
\end{pmatrix}\in GL_2(\mathbb{Z}/3\mathbb{Z})\right\} \cong \Sigma_3.$$
Then 
$$ TMF_1(3) \simeq TMF(3)^{h\Sigma}  $$ and
$$ L_{K(2)}Tmf_1(N) \simeq L_{K(2)}TMF_1(N) \simeq L_{K(2)}TMF(3)^{h\Sigma_3} \simeq (\prod_{C_i}
\widetilde{E}_2(\widehat{C_i}))^{h\Sigma_3} $$
(see \cite{mahowaldrezk}, Section 2).
\subsubsection{$K(1)- K(2)$-local $Tmf_1(N)$}
Let $\mathcal{M}_1(N)^{ss,ord} \subset \mathcal{M}_1(N)^{ss}$ denote the substack corresponding to the formal neighborhoods of the supersingular points where we removed the supersingular points, i.e. all points in the formal neighborhood with ordinary reduction. If one is inclined to drawing commutative diagrams one sees that $\mathcal{M}_1(N)^{ss,ord}$ is given by a pullback diagram of (formal) stacks
\begin{equation}\label{k1k2local}
\xymatrix{ \mathcal{M}_1(N)^{ss,ord} \ar[r] \ar[d] & \mathcal{M}_1(N)^{ss} \ar[d] \\
\mathcal{M}^{ord}_1(N) \ar[r] & \overline{\mathcal{M}_1(N)}\widehat{_p}
}.
\end{equation}
%\begin{remark}
%If one prefers to have more concrete \footnote{The reader's mileage regarding the term "concrete" may vary} objects at hand one can also see that removing the supersingular locus i.e. the vanishing locus of the Hasse invariant $A$ amounts to formally inverting its integral lift $E_{p-1}$. This procedure though destroys $p$-completeness so in order to capture the whole formal neighborhood minus the supersingular points one has to $p$-complete once again. Spelled out in formulas we get
%$$\mathcal{M}_{ell}^{ss,ord} \cong  \bigsqcup_I \text{Spf}(\mathbb{W}(k_i)((E_{p-1}))\widehat{_p}).
%$$
%\end{remark}
Applying $\mathcal{O}^{top}$ and passing to the global sections gives
$$L_{K(1)}L_{K(2)}Tmf_1(N) \simeq \Gamma(\mathcal{M}^{ss,ord}_1(N),\mathcal{O}^{top}).$$

\begin{remark}
This is not a theorem but rather part of the construction of $Tmf_1(N)$.

\end{remark}
Finally we have the following theorem (slightly adapted from \cite{Dylan} (Theorem 2.9) to our needs) which we reccord for later use:
\begin{prop}\label{k1localmaps}
We have isomorphisms 
\begin{align*}
    [KU\widehat{_p},L_{K(1)}Tmf_1(N)] & \cong Cts_{\mathbb{Z}_p^{\times}}(Cts(\mathbb{Z}_p^{\times},\mathbb{Z}_p),V(\Gamma_1(N))) \\
    [L_{K(1)}Tmf_1(N), L_{K(1)}Tmf_1(N)] & \cong Cts_{\mathbb{Z}_p^{\times}}(V(\Gamma_1(N)), V(\Gamma_1(N))).
\end{align*}
Here $Cts_{\mathbb{Z}_p^{\times}}(-,-)$ denotes continuous  $\mathbb{Z}_{p}^{\times}$-equivariant maps.
\end{prop}

\section{Generalities on $\mathbb{E}_{\infty}$-Orientations}
We are now going to collect all the basic ingredients for studying $\mathbb{E}_{\infty}$-orientations from \cite{AHR}. The intended purpose is to make this work more self-contained and we make no claim on originality apart from Proposition \ref{k2hecke} were we provide a correction to the corresponding result in \cite{AHR}, namely Proposition 4.8 and Example 4.9.
%\subsubsection{Relationship between Orientations and Genera}
%Let $f:X\rightarrow Y$
%be a proper map of smooth manifolds. Now let $f^* TY$ be the bundle of tangent vectors along the fiber.
%The Becker-Gottlieb transfer now gives a map 
%$$
%\tau(f): Y_+ \rightarrow X^{-Tf}.
%$$
%Now assume we have a generalized multiplicative cohomology theory $E$ such that the bundle $-Tf$ is oriented with respect to $E$. This means there exists a Thom class $U_E \in E^{-d} X^{-Tf}$ (here $d$ is the rank of the virtual bundle $-Tf$) inducing an isomorphism

%$$
%E^*(X_+) \cong E^{*-d}X^{-Tf}. 
%$$
%Now composing this Thom isomorphism together with the map $\tau(f)$ we get the so called "Umkehr map"
%$$
%f_!: E^*(X_+) \rightarrow E^*(Y_+)
%$$

%In particular for $X$ a compact manifold of dimension $d$ 
%\todo{Ok, here is something wrong with the  superscripts and I dunno what is Latex doing here}
\subsection{The Spectrum of Units}\label{spectrumofunits}\label{orientations}
Let $R$ be a homotopy commutative ring spectrum. The space of units $GL_1(R)$ is defined as the homotopy pullback in the following diagram
$$\xymatrix{
GL_1(R) \ar[r]\ar[d] & \Omega^{\infty}R\ar[d] \\
(\pi_0 R)^{\times} \ar[r] & \pi_0 R
}
$$
where the right vertical map is the projection onto the path components and the lower horizontal map is given by the inclusion of invertible elements. 
Now if $X$ is an unbased space, then
$$ [X,GL_1 (R)] \cong (R^0(X_{+}))^{\times}$$ 
and 
$$[Y,GL_1(R)]_{\star} \cong (1+\widetilde{R}^0(Y))^{\times} \subseteq R^0(Y_{+})^{\times} $$ 
for $Y$ a pointed space (writing $[-,-]_{\star}$ for homotopy classes of base point preserving maps between pointed spaces).

In particular for $Y \simeq S^n$ we get the homotopy groups of the space $GL_1(R)$:
$$\pi_n GL_1(R) \cong \begin{cases} (\pi_0(R))^{\times} \text{  for  } n=0 \\
\pi_n(R) \text{  for  } n\geq 1
\end{cases} $$

where the isomorphism $\pi_n GL_1(R) \cong \pi_n (R)$ is given by $x \mapsto 1+x$.

If $R$ is an $\mathbb{E}_\infty$-ring spectrum then $GL_1(R)$ turns out to be an infinite loop space, so there exists a connective spectrum $gl_1(R)$ such that $\Omega^\infty gl_1(R) \simeq GL_1(R)$.
\begin{remark}The map $GL_1(R) \rightarrow \Omega^{\infty}R$ is in general not a map of infinite loop spaces and therefore does not induce a map of spectra $gl_1 R \rightarrow R$. In particular, although the spaces $GL_1(R)$ and $\Omega^{\infty}R$ have weakly equivalent basepoint components, the natural inclusion map $GL_1(R)\hookrightarrow \Omega^{\infty} R$ is not basepoint preserving. 
\end{remark}
The functor $gl_1(-)$ has a left adjoint "up to homotopy" given by the functor $\Sigma^{\infty}_+\Omega^{\infty}$ from connected spectra to $\mathbb{E}_{\infty}$-ring spectra. More precisely we have the following:
\begin{prop}[\cite{ABGHR} Theorem 5.1]
The functors $\Sigma_+^{\infty} \Omega^{\infty}$ and $gl_1$ induce adjunctions 
\begin{equation}\label{fundamentaladjunction} 
    \Sigma^{\infty}_+\Omega^{\infty}: \text{ho}((-1)\text{connected spectra}) \rightleftarrows \text{ho}(\mathbb{E}_{\infty}\text{-ring spectra}):gl_1
\end{equation}
of categories enriched over the homotopy category of spaces. Furthermore the left adjoint preserves homotopy colimits and the right adjoint preserves homotopy limits. 
\end{prop}

\begin{remark}
The adjunction above should be thought of as an higher algebraic analogue to the adjunction 
$$ \mathbb{Z}[-]:AbGroups 
 \rightleftarrows ComRings:GL_1(-)
 $$

between the units of a commutative ring and the group ring functor. 
\end{remark}
\begin{example}
For the sphere spectrum $\mathbb{S}$ the space $GL_1(\mathbb{S})$ consists of the components of degree $\pm 1$ of the space $Q\mathbb{S}^0$. The one-fold delooping $BGL_1(\mathbb{S})$ is a classifying space for stable spherical fibrations.    
\end{example}

\begin{example}
Let $R$ be an ordinary commutative ring and $HR$ the associated Eilenberg-MacLane spectrum. Then $gl_1(HR) \simeq H(R)^{\times}$.
\end{example}
\subsection{Thom Spectra and $\mathbb{E}_{\infty}$-Orientations}\label{thomorient}
We make the following notational convention following \cite{AHR}: Infinite loop spaces will be denoted by upper case letters and the corresponding connective spectra with the corresponding lower case letters. If $g$ is a connective spectrum we will denote with $bg \overset{def}{=}\Sigma g$ its one-fold suspension.

\begin{defi}
Let  $j:g \rightarrow gl_1\mathbb{S} $ be a map of spectra with cofiber $Cj$. We define the Thom spectrum $Mf$ associated to $j$ as the homotopy pushout in the category of $\mathbb{E}_{\infty}$-ring spectra
$$\xymatrix{
\Sigma^{\infty}_+ \Omega^{\infty} gl_1 \mathbb{S}\ar[r] \ar[d] & \mathbb{S}\ar[d] \\
\Sigma^{\infty}_+\Omega^{\infty} Cj \ar[r] & Mj
}.
$$
Here the upper horizontal map is the counit of the adjunction (\ref{fundamentaladjunction}) from above.

\end{defi}
\begin{example}
The classical $J$-homomorphism $O \xrightarrow[]{J} Gl_1(\mathbb{S}) $ is a map of infinite loop spaces and therefore induces a map of spectra $\Sigma^{-1} bo\langle2 \rangle \xrightarrow{j} gl_1(\mathbb{S})$. Given a map $F=\Omega^{\infty}f:G\rightarrow O$ of infinite loop spaces we will usually write $MG$ for the Thom spectrum associated to the composite $j\circ f: g \rightarrow \Sigma^{-1} bo\langle2 \rangle \rightarrow gl_1(\mathbb{S})$ if the map $f$ is clear from the context. Take for example connective covers of $O$ such as $Spin \cong O\langle 3 \rangle$. The resulting Thom spectrum $MSpin$ coincides with the classical one which goes by the same name. Note however that in the case of connective covers $O\langle n \rangle $ and $U\langle n \rangle$ of $O$ and $U$ the standard notational conventions are inconsistent with this practice as in this case the corresponding Thom spectra are usually denoted by $MO\langle n+1 \rangle$ and $MU\langle n+1 \rangle$. So we have for example
$$ MSpin = MO \langle 4 \rangle.
$$
\end{example}
Now let $R$ be an $\mathbb{E}_{\infty}$-ring spectrum and $\iota: \mathbb{S} \rightarrow R $ the unit map. Using the adjunction (\ref{fundamentaladjunction}) together with the description of the Thom spectrum $Mj$ as a homotopy pushout, one can see that the space of $\mathbb{E}_{\infty}$ maps $Map_{\mathbb{E}_\infty} (Mj, R)$ homototpy equivalent to the homotopy pullback in the following diagram:

\begin{equation}
    \xymatrix{ Map_{\mathbb{E}_{\infty}}(Mj,R) \ar[r] \ar[d] & Map_{Spectra}(Cj, gl_1 R) \ar[d] \\
    \{gl_1(\iota)\} \ar[r] & Map_{Spectra}(gl_1\mathbb{S}, gl_1R).
    }
\end{equation}{}
This means in practice that to give an element in $\pi_0 Map_{\mathbb{E}_{\infty}}(Mj,R)$ is equivalent to finding a dotted arrow making the following diagram commute:

\begin{equation}
\xymatrix{ \Sigma^{-1} bg \ar[r]^{j} & gl_1\mathbb{S} \ar[d]_{gl_1(\iota)} \ar[r] &Cj \ar[r] \ar@{.>}[ld] & bg \\
& gl_1 R & & 
}
\end{equation}

In what follows we will replace the spectrum $gl_1(R)$ with other spectra and make the following definition:
\begin{defi}
Let $X$ be a spectrum under $gl_1\mathbb{S}$ 
$$i: gl_1\mathbb{S} \rightarrow X$$ 
and $g$ a spectrum together with a map 
$$ j: g \rightarrow gl_1\mathbb{S} $$ with cofiber $Cj$. 
The space $Orient(bg,X)$ is the homotopy pullback 
$$
\xymatrix{ 
Orient(bg,X) \ar[r] \ar[d] & Map_{Spectra}(Cj, X) \ar[d] \\
\{i \} \ar[r] & Map_{Spectra}(gl_1\mathbb{S}, X).
}
$$
\end{defi}
\begin{remark}For a Thom spectrum $MG$  corresponding to a map $g\rightarrow gl_1 \mathbb{S}$ and an $\mathbb{E}_{\infty}$ -ring spectrum $E$ the space $Orient(bg,gl_1 E)$ is nothing but the space $Map_{\mathbb{E}_{\infty}}(MG,E)$.
\end{remark}
\begin{remark}
It is clear from the definition that the functor $Orient(bg,-)$ from spectra to spaces commutes with homotopy limits. 
\end{remark}

\begin{prop}[\cite{AHR} Section 2.3]\label{torsor}
If $Orient(bg,X)$ is non-empty, then $\pi_0 Orient(bg,X)$ is a torsor for the group 
$$\pi_0 Map_{Spectra}(bg,X).$$
\end{prop}

\subsection{Rational Orientations}\label{stablemillerinvariant}
Observe that rationally for an $\mathbb{E}_{\infty}$ ring spectrum $R$ we always have an $\mathbb{E}_{\infty}$-orientation given by the composite 

$$\alpha:
MU\otimes \mathbb{Q} \rightarrow H\mathbb{Q} \simeq  \mathbb{S}\otimes \mathbb{Q} \rightarrow R \otimes \mathbb{Q}
$$
where the first map is taking the $0$-th Postnikov section of $MU$ (taking Postnikov sections induces $\mathbb{E}_{\infty}$ maps) and the second is induced by the unit map.

\subsubsection{The Rational Logarithm}
Let $R$ be a ring spectrum and $X$ a pointed, connected and finite CW complex. We then have a natural transformation

\begin{align*}
    \ell_{\mathbb{Q}}: (1+ \widetilde{R}^0(X))^\times &\rightarrow \widetilde{R}^0(X)\otimes \mathbb{Q}\\
    1+x & \mapsto \log(1+x).
\end{align*}
If $R$ is an $\mathbb{E}_{\infty}$ ring spectrum this natural transformation arises from a map of spectra 
$$ gl_1 R \langle 1\rangle \xrightarrow{\ell_{\mathbb{Q}}} R\otimes \mathbb{Q} \langle 1 \rangle 
$$
inducing the natural inclusion on homotopy groups (see \cite{AHR} Lemma 3.3). In particular if $R$ is a \emph{rational} $\mathbb{E}_{\infty}$-ring spectrum the above map is a weak equivalence.
As a rational $\mathbb{E}_{\infty}$-ring spectrum always admits a complex $\mathbb{E}_{\infty}$-orientation $\alpha$, as discussed above, we get equvialences 
$$
\pi_0 Map_{\mathbb{E}_{\infty}}(MU, R\otimes \mathbb{Q}) \cong [bu, gl_1 R\otimes \mathbb{Q} \langle 1 \rangle  ] \cong [bu , R\otimes \mathbb{Q} \langle 1 \rangle ].
$$
where the first equivalence follows from Proposition \ref{torsor} and the second from applying $\ell_{\mathbb{Q}}$.
Here we are allowed to take connective covers of $gl_1 R \otimes \mathbb{Q}$ and $R\otimes \mathbb{Q}$ in the second isomorphism as $bu $ is 1-connected.

Note that since $\pi_0 gl_1\mathbb{S}= \mathbb{Z}/2 \mathbb{Z}$ and all the higher homotopy groups of the sphere spectrum (and therefore also of $gl_1 \mathbb{S}$) are torsion we have $gl_1 \mathbb{S} \otimes \mathbb{Q} = *$. Applying this observation to the cofiber sequence 
$$ gl_1\mathbb{S} \rightarrow gl_1 \mathbb{S}\otimes \mathbb{Q} \rightarrow gl_1 \mathbb{S}\otimes \mathbb{Q}/\mathbb{Z} 
$$
we get a weak equivalence
$$ gl_1\mathbb{S}\otimes \mathbb{Q}/\mathbb{Z} \simeq bgl_1 \mathbb{S}.
$$

This allows us to define the stable Miller invariant, an object which will become usefull later on:
\begin{defi}
Given a spectrum $i: gl_1\mathbb{S} \rightarrow R$  under $gl_1\mathbb{S}$ and a map $j:b \rightarrow bgl_1\mathbb{S}$ the stable Miller invariant $m(i,j)$ is the composition
$$
b \xrightarrow{j} bgl_1 \mathbb{S} \xleftarrow{\simeq} gl_1 \mathbb{S}\otimes \mathbb{Q}/\mathbb{Z} \xrightarrow{i\otimes \mathbb{Q}/\mathbb{Z}} R \otimes \mathbb{Q}/\mathbb{Z}.
$$
\end{defi}
If the maps $j$ and $i$ are understood from the context we will write $m(b,R)$.

\subsubsection{The Hirzebruch Characteristic Series}
Let $R$ be a rational homotopy commutative ring spectrum together with an orientation
$\alpha: MU \rightarrow R$.
Let $V\rightarrow B $ be a complex vector bundle (to simplify matters with gradings we may assume of virtual dimension 0) and $B^V$ the corresponding Thom space.
An $R$-orientation of $V\rightarrow B$ is equivalent to the existence of a Thom class $U_\alpha \in R^0(B^V)$, i.e. an element which generates $R^*(B^V)$ as a free module of rank one over $R^*(B_{+})$.
The choice of such a Thom class is not unique though.
Another orientation 
$$
\beta: MU \rightarrow R
$$
gives rise to a different Thom class $ U_{\beta}$ and two such Thom classes differ by a \emph{difference class} $\delta(U_{\alpha}, U_{\beta}) \in (R^0(B_+))^{\times} \cong [B,GL_1R]$, meaning

$$
U_\alpha = \delta(U_{\beta},U_{\alpha})\cdot U_{\beta}.
$$
Hence the set of orientations of $V$ is a torsor for the group $(R^0(B_+))^{\times}$. 

Now let $L$ be the tautological line bundle over $B=\mathbb{CP}^{\infty}$ and $U_{\alpha}L$ be its standard Thom class in ordinary cohomology with corresponding Euler class $x$.
Let $\beta:MU \rightarrow R$ be another map of ring spectra with corresponding Thom class $U_{\beta}L \in R^2((\mathbb{CP})^L)$.
\begin{defi}
The Hirzebruch characteristic series of the orientation $\beta$ is the difference class 
$$K_{\beta}(x) \overset{\text{def}}{=} \delta(U_{\alpha}L, U_{\beta}L) = 1+ o(x) \in (R^0(\mathbb{CP}_+^{\infty}))^{\times} \cong H^0(\mathbb{CP}^{\infty}_+;R_*)^{\times}.
$$
If $F$ denotes the formal group law over $R_*$ classified by $\beta_*$, then 
$$K_{\beta}(x)= \frac{x}{\exp_F(x)}.
$$
\end{defi}
Let $\beta:MU\rightarrow R\otimes \mathbb{Q}$ be a (homotopy commutative) multiplicative orientation. The difference class $\delta(\alpha,\beta)$ determines a map 
$$BU\rightarrow GL_1(R)\langle 1 \rangle. 
$$
Denote with $c_\beta$ the postcomposition of this map with the rational logarithm $\ell_{\mathbb{Q}}$.
\begin{prop}[\cite{AHR} Prop. 3.12]
If $v$ denotes the periodicity element $v= 1-L \in K^0(S^2) \cong \pi_2 BU$, then 
$$ (c_{\beta})_*(v^k) = (-1)^k t_k
$$
where the $t_k\in \pi_{2k}R$ are defined by the formula
$$K_{\beta}(x) = \exp \left(\sum_{k\geq 1}\frac{t_k}{k!}x^k \right).
$$
\end{prop}
\begin{prop}[\cite{AHR} Prop.3.15, Cor. 3.16]
Let $R$ be an $\mathbb{E}_{\infty}$-ring spectrum and $\beta: MU\rightarrow R$ a homotopy commutative orientation with Hirzebruch series 
$$K_{\beta}(x) = \exp(\sum_{k\geq 1}\frac{t_k}{k!}x^k).
$$
Then $$m(bu,gl_1 R)_*v^k = (-1)^k t_k \text{ mod } \mathbb{Z}\in \pi_{2k}R\otimes \mathbb{Q}/ \mathbb{Z}.$$
If $\beta^{\prime}$ is another homotopy commutative orientation with characteristic series 
$$K_{\beta^{\prime}}(x) = \exp(\sum_{k\geq 1}\frac{t^{\prime}_k}{k!}x^k)
$$
then $ t_k \equiv t^{\prime}_k$ in $\pi_{2k}R\otimes \mathbb{Q}/\mathbb{Z}$.
\end{prop}
%Now let $\mathcal{L}$ be the tautological line bundle over $\mathbb{CP}^{\infty}$ and $x= \zeta^*(U_{\alpha})$ the Euler class (that is the pullback along the zero section $\zeta$) of its canonical Thom class $U_{\alpha}(\mathcal{L})\in H^2(\mathbb{CP}^{\infty^{\mathcal{L}}})$ in ordinary cohomology. 

\subsection{Chromatic Localizations of the Spectrum of Units}\label{chromaticlocalizationofunits}

As we remarked in Section \ref{spectrumofunits}, although there is a weak equivalence of spaces $GL_1(R)\langle 1 \rangle \simeq \Omega^{\infty}R\langle 1 \rangle $ for $R$ an $\mathbb{E}_{\infty}$-ring spectrum, this is \textbf{not} an equivalence of infinite loop spaces so we usually do \textbf{not} get a map of spectra $gl_1R \rightarrow R$. This changes however once we localize with respect to one of the Morava K-theories $K(n)$. The reason behind this lies in the existence of the Bousfield-Kuhn functors $\Phi_n$.

\begin{prop}[see \cite{bousfieldkuhnfunctor}]\label{bousfieldkuhnfunctor}
For each prime $p$ and each $n\geq 1$ there is a functor 
$$\Phi^f_n : Spaces_* \rightarrow Spectra $$ such that 
for a spectrum $X$ 
$$ L^f_{K(n)} X \simeq \Phi^f_n \Omega^{\infty}X.$$
Setting $\Phi_n = L_{K(n)}\Phi_n^{f}$ we get a natural equivalence 
$$ L_{K(n)} \simeq \Phi_n \Omega^{\infty}. $$
\end{prop}

This means in practice that $K(n)$-localization does only depend on the underlying zeroth space of a spectrum and does not take the infinite loop space structure into account. Furthermore, as $K(n)$-localization  for $n\geq 1$ does not distinguish between a spectrum $X$ and any of its connective covers $X\langle m \rangle$ (the reason being that one can succesively go up the Postnikov tower of $X$ where all fibers between succesive stages consist of Eilenberg-MacLane spectra, which are $K(n)$-acyclic by work of Ravenel and Wilson (\cite{ravenelwilson})) we have 
$$
L_{K(n)}gl_1(R) \simeq L_{K(n)}gl_1 (R)\langle 1 \rangle \simeq L_{K(n)} R \langle 1 \rangle \simeq L_{K(n)} R.
$$
Unfortunately the functor $gl_1$ does not commute with localizations, although the following theorem tells us that in positive degrees we have rather good control of the discrepancy.
\begin{prop}[\cite{AHR} Theorem 4.11]
Let $R$ be an $E_n$-local $\mathbb{E}_{\infty}$-ring spectrum. Then the homotopy fiber $F$ of the canonical map  $gl_1 R \rightarrow L_n gl_1 (R)$ is torsion, and 
$$\pi_i F = 0$$ for $i>n$. 
\end{prop}

For our purposes we will need a slight variation of this result.

\begin{coro}[\cite{Dylan} Lemma 4.3]\label{connectedness}
Let $R$ be a $p$-complete and $E_n$-local $\mathbb{E}_{\infty}$-ring spectrum and let $\widetilde{F}$ be the fiber of the canonical map 
$$ gl_1 R \rightarrow L_{K(1)\vee \ldots \vee K(n)}gl_1 R.$$
Then $\pi_i \widetilde{F}= 0$ for $i>n+1$.
\end{coro}
\begin{remark}
Corollary \ref{connectedness} is already used in \cite{AHR} although the authors never state this explicitly.
\end{remark} 
\subsection{Rezk's Logarithm}
When working $K(n)$-locally we are always implicitly working at a prime $p$ which we will suppress from notation. 
Let $R$ be a $K(n)$-local $\mathbb{E}_{\infty}$-ring spectrum.
We can precompose the weak equivalences 
$$ L_{K(n)}gl_1(R) \simeq L_{K(n)} R
$$
with the natural localization map
$$ gl_1(R) \rightarrow L_{K(n)}gl_1 R
$$
and get a "logarithmic" natural transformation 

$$ \ell_n: (1+ \widetilde{R}^0 X)^{\times} \subseteq R^0(X_+)^{\times} \rightarrow L_{K(n)}\widetilde{R}^0(X).
$$
which has been studied in detail by Rezk in \cite{Rezklogarithm}. 
\subsubsection{$K(1)$-local Logarithm}\label{k1locallog}
Let $R$ be a $K(1)$-local $\mathbb{E}_{\infty}$ ring spectrum such that the kernel of the natural map 
$$\pi_0(L_{K(1)}\mathbb{S})\rightarrow \pi_0 R  
$$
given by the unit map contains the torsion subgroup of $\pi_0 L_{K(1)}\mathbb{S}$. All spectra in this paper to which we will apply the proposition below satisfy this technical condition.
$K(1)$-local $\mathbb{E}_{\infty}$-ring spectra $R$ come equipped with two operations $\psi$ and $\theta$ (i.e. the structure of a $\theta$-algebra) such that for an element 
$x\in R^0(X)$

$$\psi(x) = x^p + p\theta(x).$$

\begin{prop}[\cite{Rezklogarithm}, Thm.1.9]
Let $R$ be a K(1)-local $\mathbb{E}_{\infty}$ ring spectrum satisfying the conditions above and $X$ a finite complex. Then the effect of $\ell_1$ on an element $x\in R^0(X_+)^{\times}$ is given by

\begin{equation} \label{k1locallog}
\ell_1(x) = \left(1 - \frac{1}{p}\psi\right) \log x = \frac{1}{p}\log \frac{x^p}{\psi(x)}
= \sum_{k=1}^{\infty}(-1)^k \frac{p^{k-1}}{k} \left(\frac{\theta(x)}{x^p}\right)^k.
\end{equation}

\end{prop}
Note that for an element of the form $x= 1+\epsilon$ such that $\epsilon^2=0$ we get the much simpler expression
$$
\ell_1 (1+\epsilon) = \epsilon - \frac{1}{p}\psi(\epsilon).
$$
This holds in pariticular for $X \simeq \mathbb{S}^n$ and hence the formula above describes the effect of the $K(1)$-local logarithm on homotopy groups. For the spectra $L_{K(1)}Tmf(\Gamma)$ thus Propostion \ref{theta} implies.
\begin{prop}\label{k1locallogtmf}
Let $g \in \overline{MF_{k}(\Gamma)}$ be a modular form of weight $k$ representing an element $1+g \in \pi_{2k}gl_1Tmf_1(N)$. Then the effect of $\ell_1$ on said element is given by 
$$\ell_1 (1+g) = (1-p^{k-1} \text{Frob})g.
$$
\end{prop}
\subsubsection{Logarithm for Morava E-Theories}
In his work Rezk also gives formulas for the $K(n)$-local logarithm for Morava $E$-theories $E_n$ associated to height $n$ formal group laws over perfect fields of characteristic $p$. For us only the case of height 2 is relevant and the following is the specialization of Theorem 1.11 of \cite{Rezk}.
\begin{prop}\label{k2locallog}
Let $E_2$ be as above and $X$ a finite complex. Then the effect of $\ell_2$ on an element $x\in E^0(X_+)^{\times}$ is given by 

$$
\ell_2(x)= \sum_{k=1}^{\infty} (-1)^{k-1} \frac{p^{k-1}}{k}M(x)^k = \frac{1}{p}\log (1+p\cdot M(x))
$$

where $M: E_2^0(X)\rightarrow E_2^0(X)$ is the unique cohomology operation such that 

$$
1+p\cdot M(x) = \prod_{j=0}^2\Big(\prod_{\overset{A\subseteq (\mathbb{Z}/p)^2}{|A|=p^j}} \psi_A (x)\Big)^{(-1)^jp^{(j-1)(j-2)/2}}.
$$
\end{prop}
The reader unfamiliar with power operations on Morava $E$-theories may want to consult \cite{Rezk} and the references therein.
%\begin{remark}
%Note that the power operation corresponding to the subgroup of order $p^2$ in the formula above corresponds to the isogeny "multiplication with $p$ " on the formal group and not to the isogeny corresponding to the twofold iterated Frobenius as claimed in \cite{yufei}. 
%\end{remark}

\begin{prop}\label{k2hecke}
Let $g\in MF_k(\Gamma)$ be a modular form of weight $k$ representing an element $1+g \in \pi_{2k}gl_1 Tmf_1(N)$. Then the effect of $\ell_2$ on this element
is given by 
$$ 1+g \mapsto (1-T_p + p^{k-1}\langle p \rangle)\Tilde{g}
$$
where $T_p$ and $\langle p \rangle$ are the Hecke and diamond operators of Section \ref{Heckeoperators} and $\Tilde{g}$ denotes the image of $g$ under the map
$$
\pi_{2k}Tmf_1(N) \rightarrow \pi_{2k}L_{K(2)}Tmf_1(N)
$$
induced by $K(2)$-localization.
\end{prop}
\begin{remark}
This formula differs from Propostion 4.8 of \cite{AHR} and Theorem 3.4 of \cite{Dylan} by inclusion of a factor $\langle p \rangle$. As we will see in the following proof this factor should actually appear, but remark also that its omission does not change the main results of those works as both deal with level structures (level 1 and $\Gamma_0(N)$ respectively) where the diamond operators act trivially. 
\end{remark}
\begin{proof} We will only give a full argument for the cases $N\geq 5$ and sketch the required modifications for $N=3$.
First recall that $L_{K(2)}Tmf_1(N)\simeq \prod_{\text{ssg. } C_i}E_2(C_i,\beta_i)$
where $E_2(C_i,\beta_i)$ is the Lubin-Tate theory corresponding to the universal deformation of the formal group $\widehat{C_i}$ of a supersingular elliptic curve $C_i$ with a chosen $\Gamma_1(N)$-level structure $\beta_i$ and the product is taken over all isomorphism classes of pairs $(C_i,\beta_i)$ of supersingular elliptic curves together with a $\Gamma_1(N)$-level structure. The famed Serre-Tate theorem states an equivalence between the category of deformations of a supersingular elliptic curve and the deformations of its formal group. In particular both the universal deformation $\widetilde{C_i }$ of $C_i$ and of its formal group $\widehat{C_i}$ are defined over the the ring $\pi_0 E_2(C_i)$. 

The effect on homotopy groups of the canonical localization map $$Tmf_1(N) \rightarrow L_{K(2)}Tmf_1(N)$$ is given by
\begin{align*}
    \overline{MF_k(\Gamma_1(N))}\cong\pi_{2k}Tmf_1(N) &\rightarrow \pi_{2k}L_{K(2)}Tmf_1(N) \cong \prod_{\text{ssg.} C_i}\pi_{2k} E_2(C_i) \\
    f &\mapsto \prod_{\text{ssg.} C_i} f(\widetilde{C_i},\omega,\beta)
    \end{align*}
Note that this map is not canonical as it involves choices of the invariant differential and the level structure. 
Propostion \ref{k2locallog} above specialized to the case $X \simeq \mathbb{S}^{2k}$ now gives us  
$$ \ell_2(x) = (1- \frac{1}{p} \sum_{i=0}^p \psi_i^p + \phi^p)x
$$
for an element $x\in \pi_{2k}E_2$. Here $\phi^p$ and $\psi_i^p$ are power operations on $E_2(C)$-theory corresponding to the isogeny of rank $p^2$ given by the map "multiplication with $p$"  and to the $p+1$ isogenies of rank $p$ on $\widehat{\mathbb{G}_{E}}$, respectively. After choosing a coordinate $x$ on $\widehat{\mathbb{G}_E}$ these isogenies correspond to the $p+1$ roots of the $p$-series $[p]_{\widehat{\mathbb{G}_E}}(x)$.

We may again apply the Serre-Tate theorem and consider isogenies of the universal deformation of the underlying supersingular elliptic curve instead. The isogeny "multiplication with $p$" sends an elliptic curve to itself and a level structure $\beta$ to $p\cdot \beta$.   An isogeny of degree $p$ on an elliptic curve $E$ is the same as quotienting out a subgroup $H$ of rank $p$. Recall from Section \ref{Heckeoperators} that this is exactly how we defined the Hecke operators $T_p$ apart from the normalization factor $p^k$ corresponding to the weight. The result now follows since the operations $\sum_0^{p}\psi_i^p$ and $\phi^p$ act on a generator $x$ (i.e a coordinate of the formal group) of $\pi_2 E_2=E_2^0(\mathbb{CP}^1)$ by multiplication with $p$.

In the case $N=3$ one uses the equivalence $Tmf_1(3) \simeq Tmf(3)^{h\Sigma_3}$. Modular forms and Hecke operators for the level structures $\Gamma(N)$ are defined completely analogously as in Section 2 (see \cite{Katzpadicproperties} Section 1) and the argument above applys verbatim. The result now follows after passing to homotopy fixed points.

\end{proof}

\section{The Space $Map_{\mathbb{E}_{\infty}}(MU,Tmf_1(N))$}
In this section we will give a description of the path components of the space $Map_{\mathbb{E}_{\infty}}(MU,Tmf_1(N)\widehat{_p})$, i.e null homotopies of the composite 

$$ \Sigma^{-1}bu \rightarrow gl_1 \mathbb{S} \xrightarrow{\iota} gl_1 Tmf_1(N)\widehat{_p}, $$ or, equivalently, lifts making the following diagram commute 

$$\xymatrix{
\Sigma^{-1} bu \ar[r] & gl_1\mathbb{S} \ar[d] \ar[r] & gl_1\mathbb{S}/u  \ar@{-->}[ld] \\
 & gl_1 Tmf_1(N)\widehat{_p} & }.$$

We will do this in several steps: In Section \ref{glsutmf} we first give a description of the set $[gl_1\mathbb{S}/u,L_{K(1)}Tmf_1(N)]$ in terms of $p$-adic measures. For $p>2$ we essentially follow the arguments in \cite{AHR} and \cite{Dylan}, for $p=2$ however some additional work is required. We then proceed to describe the spaces $Orient(bu,gl_1L_{K(1)} Tmf_1(N))$ and $Orient(bu,gl_1L_{K(1)\vee K(2)}Tmf_1(N))$. Again, apart from a minor modification for the level structures considered here, this follows exactly the lines of the aforementioned works. 
Finally in Section \ref{lifting} we will make use of results by Hopkins and Lawson \cite{hopkinslawson} to study the question which elements in $Orient(bu,L_{K(1)\vee K(2)}gl_1Tmf_1(N))$ can be lifted to elements in $Orient(bu,gl_1Tmf_1(N)\widehat{_p}) \simeq Map_{\mathbb{E}_{\infty}}(MU,Tmf_1(N)\widehat{_p})$.
\subsection{The Space $Orient(bu, L_{K(1)}gl_1Tmf_1(N))$}\label{glsutmf}

\subsubsection{$p$-adic Measures}
 We first need to recall some facts about $p$-adic measures. All the material in this section is taken from \cite{katzeisenstein} and \cite{katzeisenstein2}. Let $R$ be a $p$-adic ring and let $X$ be a compact and totally disconnected topological space. We denote with $Cts(X,R)$ the $R$-algebra of all continuous $R$-valued functions on $X$. Observe that any element in $Cts(X,R)$ is the uniform limit of locally constant functions, so we have 

$$ Cts(X,R) = Cts(X,\mathbb{Z}_p )\hat{\otimes}_{\mathbb{Z}_p} R= \varprojlim_{n} Cts(X,\mathbb{Z}_p )\hat{\otimes}_{\mathbb{Z}_p} (R/p^nR). $$
\begin{defi}
A measure $\mu$ on $X$ with values in $R$ is a continuous $R$-linear map from $Cts(X,R)$ to $R$ (we do not assume it to be a ring homomorphism) or equivalently a continuous $\mathbb{Z}_p$-linear map from $Cts(X,\mathbb{Z}_p)$ to $R$ .

\end{defi}

For a continuous function $f:X\rightarrow R$ on $X$ we will denote its image $\mu(f) \in R$ by
$$ \int_{X} f d\mu \text{    or   } \int_{X} f(x)d \mu(x).
$$ 
\\ 
For any integer $n\geq 0$ let
\begin{align*}\dbinom{x}{n}= 
\begin{cases}
  1  &\text{ if } n=0\\
  \dfrac{x(x-1)\cdot \ldots \cdot (x-(n-1))}{n!} &\text{  if } n>0
\end{cases}
\end{align*}

be the binominal coefficient function. It maps positive integers to postive integers and therefore, by continuity, $\mathbb{Z}_p$ to $\mathbb{Z}_p$. 

Mahler's theorem (\cite{mahler}) now tells us that any continous function $f\in Cts(\mathbb{Z}_p, R)$ can be interpolated uniquely via the following expansion

$$  
f(x) = \sum (\Delta^nf)(0) \binom{x}{n}.
$$
where $\Delta$ is the shift operator defined as
$ \Delta f(x) = f(x+1) - f(x)$.

\begin{remark} It is not obvious that the numbers $ (\Delta^n f) (0)$ converge $p$-adically to zero. This is the heart of Mahler's theorem.
\end{remark}

So in other words continous functions on $\mathbb{Z}_p$ with values in $R$ can be understood as sequences $
\{a_n\}_{n\geq0}$ of elements in $R$ which converge to $0$ as $n\rightarrow \infty$.

It follows that the values on the functions $\binom{x}{n}$ uniquely determine a measure on $\mathbb{Z}_p$ with values in $R$ and conversely for any given sequence $\{b_n\}_{n\geq 0}$ of elements in $R$ there is a unique measure on $\mathbb{Z}_p$
whose value on $\binom{x}{n}$ is $b_n$.

If $R$ is flat over $\mathbb{Z}_p$, then a measure is also uniquely determined by its \emph{moments} 

$$ \int_{X} x^n d\mu.
$$

One can not however arbitrarily prescribe these moments:
Let $n\geq 0$ and $A_n$ the set of polynominals $$
h(x) = \sum a_k x^k \in R[\frac{1}{p}][x]
$$
such that $h(x) \in R$ for any $x\in \mathbb{Z}_p$. We say a sequence $\{b_n\} \in R$ satisfies the \emph{generalized Kummer congruences} if for all $h(x) =\sum a_kx^k \in A_n$ 
$$
\sum a_n b_n \in R
$$
holds.

\begin{lemma}\label{kummercongruences}
If $R$ is a $p$-adic ring flat over $\mathbb{Z}_p$, then a sequence $\{b_n\}$ of elements in $R$ arises as the moments of a (necessarily unique) measure on $\mathbb{Z}_p$ with values in $R$ if and only if the $b_n$ satisfy the generalized Kummer congruences.
\end{lemma}
\begin{proof}
This reduces to the case $R=\mathbb{Z}_p$. Note that in this case $A$ is by Mahler's Theorem the subspace of $Cts(\mathbb{Z}_p, \mathbb{Z}_p)$ spanned by the polynominals  $\binom{x}{n} $. This subspace is dense in $Cts(\mathbb{Z}_p, \mathbb{Z}_p)$ and the result follows.
\end{proof}

\subsubsection{$[KU\widehat{_p},L_{K(1)}Tmf_1(N)]$}
For a general introduction to the connection between the theory of $p$-adic measures and $K(1)$-local stable homotopy theory we refer the reader to Chapter 9 of \cite{AHR}. Relevant for us is the following result which is well known to the experts. For a proof see \cite{Dylan}.

\begin{prop}\begin{enumerate} \item
We have isomorphisms 
$$[KU\widehat{_p}, L_{K(1)}(K\widehat{_p}\wedge Tmf_1(N)]) \cong Cts(Cts(\mathbb{Z}_p^{\times},\mathbb{Z}_p), V(\Gamma_1(N) ) $$
$$ [KU\widehat{_p}, L_{K(1)}Tmf_1(N)] \cong Cts_{\mathbb{Z}_p^{\times}}(Cts(\mathbb{Z}_p^{\times},\mathbb{Z}_p), V(\Gamma_1(N))). $$
Here $Cts_{\mathbb{Z}_p^{\times}}(-,-)$ denotes $\mathbb{Z}_p^{\times}$- equivariant continous maps.
\item Let $f:KU\widehat{_p} \rightarrow L_{K(1)}Tmf_1(N)$ and $f_*=\mu \in [KU\widehat{_p}, L_{K(1)}Tmf_1(N)]$ the corresponding measure. 
Then $$ \pi_{2k} f= \int_{\mathbb{Z}_p^{\times}} x^k d\mu. $$ \end{enumerate}
\end{prop}
Note that the rings $V(\Gamma_1(N))$ and its subrings $\overline{Mf_{p,*}(\Gamma_1(N))}$ are flat over $\mathbb{Z}_p$ so a measure with values in either one of these rings is determined by its moments. Hence we get the following:
\begin{lemma}[compare \cite{Dylan} Theorem 2.7 ]
Through evaluation at a generator of $\pi_{2k}KU$ we get an injection 
$$[KU\widehat{_p},L_{K(1)}Tmf_1(N)] \hookrightarrow \prod \overline{MF_{p,k}}(\mathbb{Q}_p)  $$ 
onto the sequences of $p$-adic modular forms over $\mathbb{Q}_p$ which satisfy the generalized Kummer congruences.
\end{lemma}
%\begin{prop}
%Let $n\geq 0$. The image of the natural map 

%$$ [KU\widehat{_p},L_{K(1)}Tmf(\Gamma_1(N))] \rightarrow [KU\widehat{_p}, L_{K(1)}Tmf(\Gamma_1(N)) \otimes \mathbb{Q}] \cong \prod_{k\geq n} MF_{p,n}(\Gamma_1(N))\otimes \mathbb{Q}
%$$
%$$ g\mapsto \{\pi_{2k} g\}_{k\geq n}
%$$
%is injective with image the sequences $\{ g_n\}$ of p-adic modular forms satisfying the generalized Kummer congruences.
%\end{prop}
\subsubsection{The Adams Conjecture $K(1)$-locally}\label{OrientsKtheory}

Let $c$ be a topological generator for the group $\mathbb{Z}_p^{\times}$. The unstable complex Adams conjecture (first proved by Quillen in \cite{quillenadams}) states that the composite
$$ BU_{(p)} \xrightarrow{ 1 - \psi^c} BU_{(p)} \xrightarrow  BGL_1 \mathbb{S} $$
is null homotopic. 
Let $JU_c$ be the homotopy fiber of the map 
$$
BU_{(p)}\xrightarrow{1-\psi^c}BU_{(p)}.
$$
By standard arguments we get a commutative diagram

\begin{equation}\label{unstableadams}
\xymatrix{
JU_c \ar[r]\ar[d]^{B_c} & BU_{(p)} \ar[d]^{A_c} \ar[r]^{1-\psi^c} & BU_{(p)} \ar@{=}[d] \\
GL_1\mathbb{S}_{(p)} \ar[r] & GL_1\mathbb{S}/U_{(p)} \ar[r] & BU_{(p)} \ar[r] & BGL_1\mathbb{S}
}    
\end{equation}
where $GL_1\mathbb{S}_{(p)}/U$ is the cofiber of the complex $J$-homomorphism

$$U_{(p)} \xrightarrow{J} GL_{1}\mathbb{S}_{(p)}.$$
There is a completely analogous result for the real map $BO_{(p)} \xrightarrow{1-\psi^c} BO_{(p)}$ where in the statement of the result above every instance of the letter $U$ just needs to be replaced with the letter $O$.

Note that although $BO$, $BU$ and $BGL_1 S$ all are infinite loop spaces, only in the complex case it is known that the null homotopy given by the Adams conjecture lifts to an infinite loop map.
This is of no concern to us once we apply the Bousfield-Kuhn functor of Proposition \ref{bousfieldkuhnfunctor} (which factors through the category of spaces) to the diagram above to get the following commutative diagram 

\begin{equation} \label{ku}
    \xymatrix{
    L_{K(1)}ju_c \ar[r]\ar[d]^{\Phi_1 B_c} & KU\widehat{_p} \ar[r]^{1-\psi^c} \ar[d]^{\Phi_1 A_c} & KU\widehat{_p} \ar@{=}[d] & \\
    L_{K(1)}gl_1\mathbb{S} \ar[r]& L_{K(1)}gl_1 \mathbb{S}/u  \ar[r]& KU\widehat{_p} \ar[r]& bgl_1\mathbb{S}
    }.
\end{equation}
The real case follows completely analogously by again replacing "$u$"s with "$o$"s.
Here we used that $L_{K(1)}BU_{(p)} \simeq KU\widehat{_p}$ and similarily $L_{K(1)} BO_{(p)} \simeq KO\widehat{_p}$. 

Mahowald's calculations for the $K(1)$-local sphere (see \cite{ravenel}) imply:
\begin{prop}
For $p>2$ there is a weak equivalence 
$$ L_{K(1)} ju_c \simeq L_{K(1)}\mathbb{S}
$$
and similarily for all primes 
$$ L_{K(1)} jo_c \simeq L_{K(1)}\mathbb{S}.
$$
\end{prop}
\begin{coro}
For $p>2$ in the complex case and for all primes in the real case are the maps $\Phi_1 A_c$ and $\Phi_1 B_c$ weak equivalences.
\end{coro}
\begin{coro}
Let $n\geq 0$ and $p>2$. 
Then we have an equivalence 
$$[gl_1\mathbb{S}/u, L_{K(1)}Tmf_1(N)] \xrightarrow{\cong} [KU\widehat{_p},L_{K(1)}Tmf_1(N)]. 
$$
\end{coro}

At the prime $2$ the following complication arises: the map 
$$ \Phi_1 A_c : L_{K(1)}ju_c \rightarrow L_{K(1)}gl_1 \mathbb{S}$$ 
is not a weak equivalence anymore and hence neither is $\Phi_1 B_c$.

For the rest of this section it is understood that we are working at the prime $p=2$.

It is clear that every element $g \in [gl_1\mathbb{S}/u, L_{K(1)}Tmf_1(N)]$ gives rise to an element in $[KU\widehat{_p},L_{K(1)}Tmf_1(N)]$ by precomposition with the map $\Phi_1A_c$. We would now like to determine which elements in the later set factor through the first. In hindsight we would like to show a little bit more: It turns out we will not only need the elements in $[KU\widehat{_p},L_{K(1)} gl_1 Tmf_1(N)]$ which factor through $gl_1\mathbb{S}/u$ but those which additionally make the following diagram commute:
$$  \xymatrix{
L_{K(1)}ju_3 \ar[r] \ar[d]^{\Phi_1 A_3} & KU\widehat{_2} \ar[d]^{\Phi_1 B_3}   \\
L_{K(1)} gl_1 \mathbb{S} \ar[r] \ar[rd]^{L_{K(1)}gl_1 \iota} & L_{K(1)} gl_1 \mathbb{S}/u \ar@{-->}[d] \\ & L_{K(1)}gl_1 E\\
 }
$$

As this question might be of interest not just for the spectra $L_{K(1)}Tmf_1(N)$, we will work in greater generality than strictly needed for our purposes:
 
For the rest of this section let $E$ be a complex orientable $K(1)$-local $\mathbb{E}_{\infty}$-ring spectrum such that:
\begin{enumerate}
    \item $\pi_* E$ is concentrated in even degrees and torsion-free.
    \item $\pi_*F(KU\widehat{_2},E)$ is concentrated in even degrees and torsion-free.
\end{enumerate}
First recall that the spectrum $gl_1\mathbb{S}$ is torsion so the composite 
$$ gl_1\mathbb{S} \xrightarrow{gl_1 \iota} gl_1 E  \rightarrow gl_1 E\otimes \mathbb{Q} $$ is null homotopic and we have a canonical factorizaton
$$ \xymatrix{
gl_1 \mathbb{S} \ar[r]^{gl_1 \iota} \ar[d] & gl_1 E  \\
\Sigma^{-1} gl_1 E \otimes \mathbb{Q}/\mathbb{Z} \ar[ur] &
} $$
 
Using this observation together with the $K(1)$-local logarithm we may extend Diagram \ref{ku}
to get the following picture:

$$\xymatrix{
L_{K(1)}ju_3 \ar[r] \ar[d]^{\Phi_1 A_3} & KU\widehat{_2} \ar[d]^{\Phi_1 B_3}   \\
L_{K(1)} gl_1 \mathbb{S} \ar[r]\ar[d] \ar[rd]^{L_{K(1)}gl_1 \iota} & L_{K(1)} gl_1 \mathbb{S}/u \\
\Sigma^{-1} L_{K(1)}gl_1 E \otimes \mathbb{Q}/\mathbb{Z} \ar[r] \ar[d] & L_{K(1)}gl_1 E\ar[d]_{\ell_1}^{\simeq}\\
\Sigma^{-1} E\otimes \mathbb{Q}/\mathbb{Z} \ar[r] & E 
 }
 .$$
Note that Diagram \ref{ku} is a homotopy pushout, so a map $f: KU\widehat{_2}\rightarrow E$ factors through $L_{K(1)}gl_1 \mathbb{S}/u$ if and only it makes the following diagram commute:
$$\xymatrix{
L_{K(1)}ju_3 \ar[r] \ar[d]^{\Phi_1 A_3} & KU\widehat{_2} \ar[ddd]^{f} \\
L_{K(1)} gl_1 \mathbb{S} \ar[d] &  \\
\Sigma^{-1}L_{K(1)}gl_1 E \otimes \mathbb{Q}/\mathbb{Z}  \ar[d] & \\
\Sigma^{-1} E\otimes \mathbb{Q}/\mathbb{Z} \ar[r] & E
 }
 $$
As the spectrum $L_{K(1)}ju_3$ is equivalent to the $K(1)$-localization of $\Sigma^{-2} \mathbb{CP}^2$ its $E$-cohomology is torsion-free and concentrated in even degrees. We therefore have an inclusion $$[L_{K(1)} ju_3, E] \hookrightarrow [L_{K(1)} ju_3, E\otimes \mathbb{Q}],$$ 
similarily for $E^*KU\widehat{_2}$ by assumption.
Hence in order to check if a map $f$ makes the diagram commute it suffices to verify this after rationalization which in turn may be done on the level of homotopy groups.
As $L_{K(1)}ju_3$ is given as the homotopy fiber of the map $KU\widehat{_2}\xrightarrow{1-\psi^3} KU\widehat{_2}$ we immediately see that its homotopy groups vanish in even degrees (except in degree zero where it is isomorphic to $\mathbb{Z}_2$) and in odd degrees are given by $\pi_{2n-1} L_{K(1)} ju_3 \cong \mathbb{Z}_2/(1-3^n)$. 
As $E_*$ is concentrated in even degrees, by assumption, the only relevant group to check is $\pi_0$.
The induced map $\pi_0 L_{K(1)}ju_c \cong \mathbb{Z}_2 \rightarrow \pi_0 KU\widehat{_p}\cong \mathbb{Z}_2$ is the identity so to understand the necessary condition on $\pi_0f$ 
we now only have to understand what happens on the left hand side of the diagram. The effect of $\ell_1 \circ L_{K(1)} gl_1 \iota$ on $\pi_0$ being the identity, the critical map in question is $\Phi_1 A_3$.

Observe that 

$$\xymatrix{
L_{K(1)}gl_1 \mathbb{S} \ar@{=}[d] \ar[r] & L_{K(1)} gl_1 \mathbb{S}/u \ar[r] \ar[d] & KU\widehat{_2} \ar[d]^{r} \\
L_{K(1)} gl_1 \mathbb{S} \ar[r] & L_{K(1)} gl_1 \mathbb{S}/o \ar[r] & KO\widehat{_2}
}
$$
where the right hand side is the "realifiaction" map from complex to real K-theory, and 

$$\xymatrix{ 
L_{K(1)} \mathbb{S} \ar[r] \ar[d]_{\Phi_1 A_3^r} & KO\widehat{_2} \ar[r]^{1-\psi^3} \ar[d]_{\Phi_1 B_3^r} & KO\widehat{_2} \ar@{=}[d] \\
L_{K(1)} gl_1 \mathbb{S} \ar[r] & L_{K(1)}gl_1 \mathbb{S}/o \ar[r] & KO\widehat{_2} 
}
$$
commute.
Here the maps $\Phi_1 A_3^r$ and $\Phi_1 B_3^r$ are the real analogues of our corresponding maps with the same labels save the superscript "$r$". As discussed previously (see also \cite{AHR} Lemma 7.9) these maps are weak equivalences.
\begin{prop}\label{effectonpi}
The effect of $\Phi_1 A_3^r$ on $\pi_0$ is given by multiplication with $\frac{1}{4}\log(3)$.
\end{prop}
\begin{proof}
This is just a reformulation of Proposition 7.15 of \cite{AHR}.
\end{proof}
Combining these two diagrams together with Diagram \ref{ku} 
we get the larger commutative diagram

$$
\xymatrix{ 
L_{K(1)}ju_2 \ar[d]_{\Phi_1 A_3} \ar[r] &  KU\widehat{_2} \ar[d]_{\Phi_1 B_3} \ar[r]^{1-\psi^3} & KU\widehat{_2} \ar@{=}[d] \\
L_{K(1)}gl_1 \mathbb{S} \ar@{=}[d] \ar[r] & L_{K(1)}gl_1 \mathbb{S}/u \ar[d] \ar[r] & KU\widehat{_2}\ar[d]^{r} \\
 L_{K(1)}gl_1 \mathbb{S} \ar[d]_{(\Phi_1 A_3^r)^{-1}}^{\simeq} \ar[r] & L_{K(1)}gl_1 \mathbb{S}/o 
 \ar[d]_{(\Phi_1 B^{r}_{3})^{-1}}^{\simeq} \ar[r]& KO\widehat{_2}\ar@{=}[d] \\
L_{K(1)}\mathbb{S} \ar[r] & KO\widehat{_2} 
\ar[r]^{1-\psi^3} & KO\widehat{_2}
 }
$$

Note that the vertical maps compose to give the diagram

$$ \xymatrix{ L_{K(1)} ju_3 \ar[r] \ar[d]^{r} & KU\widehat{_2} \ar[r]^{1-\psi^3} \ar[d]^{r} & KU\widehat{_2} \ar[d]^{r} \\
L_{K(1)} \mathbb{S} \ar[r] & KO\widehat{_2} \ar[r]^{1-\psi^3} & KO\widehat{_2} 
}.
$$
Here the vertical maps $r$ are again the realification maps and their effect on $\pi_0$ is given by multiplication with $2$. 

This observation combined with Proposition \ref{effectonpi} gives the following theorem.

\begin{theo}
The effect of $\Phi_1 A_3$ is given by multiplication with $\frac{1}{2}\log(3)$. 
\end{theo}

To summarize our discussion:
\begin{lemma}\label{factorization}Let $E$ be as above.
A map $f:KU\widehat{_2} \rightarrow E$ factors through $L_{K(1)} gl_1 \mathbb{S}/u$ such that the diagram $$  \xymatrix{
L_{K(1)}ju_3 \ar[r] \ar[d]^{\Phi_1 A_3} & KU\widehat{_2} \ar[d]^{\Phi_1 B_3} \ar@/^3pc/[dd]^{f}  \\
L_{K(1)} gl_1 \mathbb{S} \ar[r] \ar[rd]^{L_{K(1)}gl_1 \iota} & L_{K(1)} gl_1 \mathbb{S}/u \ar@{-->}[d] \\ & L_{K(1)}gl_1 E\\
 }
$$
commutes 
if and only if its effect on $\pi_0$ is given by multiplication with $\frac{1}{2}\log (3)$.
\end{lemma}
%\begin{prop}
%We have an inclusion 
%$$ [KU\widehat{_p},L_{K(1)}Tmf(\Gamma)] \hookrightarrow %[KU\widehat{_p},L_{K(1)}Tmf(\Gamma)\otimes \mathbb{Q}] \cong \prod_{k} MF_{p,k} \otimes %\mathbb{Q}
%$$
%$$f \mapsto \pi_{2k}f \overset{\text{def}}{=} g_k
%$$
%onto the sequences $\{g_k \}$ of $2$-adic modular forms such that the $ g_k$ satisfy the %generalized Kummer congruences.
%\end{prop}

Note that in particular $L_{K(1)}Tmf_1(N)$ satisfies the conditions of Lemma \ref{factorization}.
The discussion above now allows us to study complex orientations without restrictions on the prime. From here on now the arguments are exactly the same as in \cite{Dylan} for $MString$-orientations of $Tmf$ with level structure. We recall the argument for the reader's convenience.

\begin{theo}\label{k1tmforient}
We have an injection 
$$ \pi_0 Orient(bu, L_{K(1)}Tmf_1(N)) \hookrightarrow [bu, L_{K(1)}Tmf_1(N)\otimes \mathbb{Q}] \cong \prod_{k\geq 1} MF_{p,k}\otimes \mathbb{Q}
$$
$$ f \mapsto \pi_{2k} f $$ onto the set of sequences $\{ g_k\}$ of p-adic modular forms such that \begin{enumerate}
    \item for all $\mathbb{c}\in \mathbb{Z}_p^\times$ the sequence $\{(1-c^k)(1-p^{k-1}\text{Frob})g_k \}$ satisfies the generalized Kummer congruences
    \item for all $c\in \mathbb{Z}_p^\times$  $\lim_{k \to \infty} (1-c^k)(1-p^{k-1} \text{Frob})g_k = \frac{1}{p}\log (c^{p-1
    })$.
\end{enumerate}
\end{theo}

\begin{proof}
We are looking for maps $\alpha$ and $\beta$ making the following diagram commute:
\begin{picture}(400,150)
\put(180,0){\makebox(0,150){
$$\xymatrix@C=1em{
L_{K(1)}ju_c \ar[r] \ar[d]_{\Phi A_c^r} & KU\widehat{_p}  \ar[r]^{1-\psi^c} \ar[d]_{\Phi_1 B_c} & KU\widehat{_p} \ar@{=}[d] & \\
L_{K(1)}gl_1 \mathbb{S} \ar[r] \ar[d] & L_{K(1)}gl_1 \mathbb{S}/u\ar[r] \ar@{-->}[d]_{\alpha} & KU\widehat{_p} \ar[rd] \ar@{-->}[d]_{\beta} \ar[r]& L_{K(1)} bgl_1 \mathbb{S} \ar[d]\\
\Sigma^{-1}L_{K(1)} gl_1 Tmf_1(N)\widehat{_p} \otimes \mathbb{Q}/\mathbb{Z} \ar[r] \ar[d]^{\ell_1} & L_{K(1)} gl_1 Tmf_1(N)\widehat{_p} \ar[d]_{\ell_1}^{\simeq} \ar[r] & L_{K(1)} gl_1 Tmf_1(N)\widehat{_p} \otimes \mathbb{Q} \ar[d]^{\ell_1} \ar[r]& L_{K(1)} gl_1 Tmf_1(N)\widehat{_p} \otimes \mathbb{Q}/ \mathbb{Z} \ar[d]^{\ell_1}\\
\Sigma^{-1} Tmf_1(N)\widehat{_p} \otimes\mathbb{Q}/\mathbb{Z} \ar[r] & L_{K(1)}Tmf_1(N)\widehat{_p} \ar[r] & L_{(K(1)}Tmf_1(N)\widehat{_p} \otimes \mathbb{Q} \ar[r] &
L_{K(1)}Tmf_1(N)\widehat{_p}\otimes \mathbb{Q}/ \mathbb{Z}
}
$$
}}
\end{picture}

Here the diagonal map is given by the Miller invariant
$$m(j,gl_1 L_{K(1)}Tmf_1(N)\widehat{_p}).
$$

Thus to give a map 
$$\alpha: L_{K(1)}gl_1 \mathbb{S}/u \rightarrow L_{K(1)}gl_1 Tmf_1(N)\widehat{_p}
$$
making the diagram commute is equivalent to giving an element $\ell_1 \circ \alpha \circ \Phi_1 B_c \in [KU\widehat{_p},L_{K(1)}Tmf(\Gamma)\widehat{_p}]$.
Let 
$$ \{ t(\alpha)_k \} \overset{\text{def}}{=} \pi_{2k}(\ell_1 \circ \alpha \circ \Phi_c B_c) \in \prod MF_{p,k} 
$$
be the sequence corresponding to $\alpha$ and 
$$\{ b(\alpha)_k\} \overset{\text{def}}{=}\{ \pi_{2k} \beta \} \in [bsu, gl_1 L_{K(1)} Tmf_1(N) \otimes \mathbb{Q}]
$$
the sequence corresponding to the map $\beta$. 
To see how these two sequences are related we will study the effect of $\beta$ on a class in $\pi_{2k} bu$ by tracing it through the following steps:
\begin{enumerate}
    \item Push it forward to the $K(1)$-localization $KU\widehat{_p}$.
    \item Pull it back along the map $1-\psi^c$ which acts by multiplication with $1-c^{k}$ on a class in $\pi_{2k}$.
    \item Map it down to $L_{K(1)}Tmf_1(N)$ along the composite 
     $\ell_1 \circ \alpha \circ \Phi_1 B_c
    $.
 \item Include it into the rationalization $L_{K(1)}Tmf_1(N)\otimes \mathbb{Q}$. We are allowed to do this without losing any information as $[KU\widehat{_p},L_{K(1)}Tmf_1(N)]$ is torsion free.
     \item Pull it back along the logarithm to $gl_1 L_{K(1)}Tmf_1(N)\otimes \mathbb{Q}$ which acts by multiplication with the factor $(1-p^{k-1}\text{Frob})$.
\end{enumerate}
Combining all these steps together we find that 
$$
t_k(\alpha) = (1-c^k)(1-p^{k-1}\text{Frob}) b_k(\alpha).
$$
We are thus reduced to verifying that $\alpha$ makes the diagram 
$$ \xymatrix{ L_{K(1)}gl_1\mathbb{S} \ar[r]\ar[d] & L_{K(1)}gl_1\mathbb{S}/u \ar[d]^{\alpha} \\
\Sigma^{-1}L_{K(1)} gl_1 Tmf_1(N)\widehat{_p} \otimes \mathbb{Q}/\mathbb{Z} \ar[r] & L_{K(1)}gl_1Tmf_1(N)
} $$ commute. We showed in Lemma \ref{factorization} above for $p=2$ that this holds if and only if $\pi_0 \alpha(1)= \frac{1}{p}\log(c^{p-1})$. The proof for $p>2$ is completely analogous, but easier (compare \cite{AHR} Proposition 14.6). 
\end{proof}
\subsection{The Space $Orient(bu, L_{K(1)\vee K(2)}gl_1Tmf_1(N))$}
We would now like to understand the space $Orient(bu,L_{K(1)\vee K(2)}gl_1 Tmf_1(N))$ which is given by the following homotopy pushout 
$$
\xymatrix{Orient(bu ,L_{K(1)\vee K(2)}gl_1 Tmf_1(N)) \ar[r] \ar[d] & Orient(bu,L_{K(1)} gl_1 Tmf_1(N)) \ar[d] \\
Orient(bu, L_{K(2)} gl_1 Tmf_1(N) \ar[r] & Orient(bu,L_{K(1)}L_{K(2)} gl_1 Tmf_1(N))}.
$$
Note that the spectrum $bu$ is $K(2)$-acyclic, hence the lower left corner in the diagram is contractible. So the diagram degenerates on path components into the following sequence

$$\xymatrix{
\pi_0 Orient(bu, L_{K(1) \vee K(2)} gl_1 Tmf_1(N) )\ar[r] & \pi_0 Orient(bu,L_{K(1)} gl_1 Tmf_1(N))\ar[d]\\
& \pi_0 Orient(bu, L_{K(1)}L_{K(2)} gl_1 Tmf_1(N))}
$$
which is "exact in the middle " in the sense that the image of the first map gets mapped by the second map to the element $*$.

Using logarithmic cohomology operations we may extend the chromatic fractures square determining $L_{K(1)\vee K(2)}gl_1Tmf_1(N)\widehat{_p}$ as follows:
\begin{equation}\label{factoringa}
    \xymatrix{ L_{K(1)\vee K(2)}gl_1 Tmf_1(N) \ar[r] \ar[d] & L_{K(2)}gl_1 Tmf_1(N)\ar[d] \ar[r]^{\ell_2}_{\simeq} & L_{K(2)}Tmf_1(N) \ar[dd] \\
L_{K(1)}gl_1 Tmf_1(N)\ar[d]^{\ell_1}_{\simeq} \ar[r] & L_{K(1)}L_{K(2)} gl_1Tmf_1(N) \ar[rd]^{L_{K(1)}\ell_2}_{\simeq} & \\
L_{K(1)}Tmf_1(N) \ar[rr]^{a} & & L_{K(1)}L_{K(2)} Tmf_1(N)
}
\end{equation}
Here the vertical arrow on the very right is just the usual $K(1)$-localization (this follows from the naturality of localizations and of the logarithm) but the lower horizontal arrow $a$ is \textbf{not} the usual localization map! Understanding this map is key to the construction of any $\mathbb{E}_{\infty}$-orientation of the spectrum of topological modular forms or versions of it.
First observe the following fact:
\begin{prop}[\cite{Dylan} Cor. 3.12]
Evaluation on $\pi_*$ gives an injection 
$$[L_{K(1)}Tmf_1(N), L_{K(1)}L_{K(2)}Tmf_1(N)] \hookrightarrow \text{Hom}(\pi_*L_{K(1)} Tmf_1(N)\otimes \mathbb{Q}, \pi_* L_{K(1)}L_{K(2)}Tmf_1(N)\otimes \mathbb{Q}). 
$$
\end{prop}
This follows immmediately from the fact that the left hand side is torsion free (see Propostion \ref{k1localmaps}). In particular this implies that the map $a$ is determined by its effect on homotopy groups.

\begin{lemma}

There exist maps $$U^{top}:L_{K(1)}Tmf_1(N) \rightarrow L_{K(1)}Tmf_1(N)$$ and $$\langle p \rangle^{top}: L_{K(1)}Tmf_1(N) \rightarrow L_{K(1)}Tmf(\Gamma)$$  such that their effects on elements $g\in \pi_{2k} L_{K(1)}Tmf_1(N)\cong MF_{p,k}(\Gamma_1(N))$ is given by the operators $U$ and $\langle p \rangle$ of Section \ref{atkinsoperator}.
\end{lemma}
\begin{proof}
This can be deduced directly from Propostion \ref{k1localmaps} as both $U$ and $\langle p \rangle$ commute with the diamond operators $\langle c,1 \rangle$  for $c\in \mathbb{Z}_p^{\times}$ (see Section \ref{padicmodularfunctions}).
\end{proof}
Recall from Proposition \ref{k2hecke} the effect of the $K(2)$-local logarithm on homotopy groups.
Combining this with the fact that $$U\circ \text{Frob}=id$$ and $$T_p = p^{k-1}\text{Frob} + \langle p \rangle U $$
(see Section \ref{padicmodularfunctions}) we get 
$$1-T_p + p^{k-1}\langle p \rangle = (1-\langle p \rangle U)(1-p^{k-1}\text{Frob}).$$ Combining this all together with the fact that the localization maps 
$$
L_{K(1)}Tmf_1(N) \rightarrow L_{K(1)}L_{K(2)}Tmf_1(N)
$$
and
$$ L_{K(2)}Tmf_1(N) \rightarrow L_{K(1)}L_{K(2)}Tmf_1(N)
$$
are injective on homotopy groups 
we get that the map $a$ (recall that it is completely determined by its effect on homotopy groups) factors as 
$$ L_{K(1)}Tmf_1(N) \xrightarrow{1-\langle p \rangle^{top} U^{top}} L_{K(1)}Tmf_1(N) \rightarrow 
L_{K(1)}L_{K(2)}Tmf_1(N)
$$
and hence we get the following commutative diagram:
\begin{equation}\label{diagramsplittingatkins}
\xymatrix{
L_{K(1)\vee K(2)}gl_1 Tmf_1(N)\widehat{_p} \ar[d] \ar[rr] & & L_{K(2)}Tmf_1(N) \ar[d] \\
L_{K(1)}Tmf_1(N) \ar[r]^{1-\langle p \rangle^{top} U^{top}} & L_{K(1)}Tmf_1(N)\ar[r] & L_{K(1)}L_{K(2)}Tmf_1(N)
}
\end{equation}
Note that the lower right horizontal and the right vertical map are just the usual ones in the chromatic fracture square for $Tmf_1(N)\widehat{_p}$. So an element in $\pi_0 Orient(bu, L_{K(1)}gl_1 Tmf_1(N))$ specified by a sequence of $p$-adic modular forms $\{g_k \}$  ($k\geq 1$) as in Theorem \ref{k1tmforient} lifts to an element in $$\pi_0 Orient(bu, L_{K(1)\vee K(2)}gl_1 Tmf_1(N)) $$ if and only if for all $k\geq 1$
$$
(1-p^{k-1}\text{Frob})g_k
$$
is in the kernel of the operator $1-\langle p \rangle U$.
\begin{theo}\label{orienttmfgamma}
We have an injection 
\begin{align*}
    \pi_0 \text{Orient}(bu, L_{K(1)\vee K(2)}Tmf_1(N) \hookrightarrow & [bu, L_{K(1)}Tmf_1(N)\otimes \mathbb{Q}] \\  \cong &\prod_{k\geq n} MF_{p,k}\otimes \mathbb{Q} \\
f  \mapsto & \pi_{2k}f
\end{align*}
onto the set of sequences $\{ g_k \}$ of $p$-adic modular forms such that \begin{enumerate}
    \item for all $c\in \mathbb{Z}_p^{\times}$ the sequence $\{(1-c^k)(1-p^{k-1}\text{Frob})g_k \}$ satisfies the generalized Kummer congruences;
    \item for all $c\in \mathbb{Z}_p^{\times}$ $\lim_{k \to \infty}(1-c^k)(1-p^{k-1}\text{Frob})g_k = \frac{1}{p}\log(c^{p-1})$; 
    \item The elements $(1-p^{k-1}\text{Frob})g_k$ are in the kernel of the operator $1-\langle p \rangle U$.
\end{enumerate}
\end{theo}

\begin{remark}
Hecke and Atkins operators as operations on elliptic cohomology theories have been studied already from the very early days of the field (see for example \cite{bakeratkins} and \cite{bakerhecke} ).
\end{remark}

\begin{remark}\label{tmflogarithm}
Note that from Diagram \ref{diagramsplittingatkins} above
and the universal property of homotopy pullbacks we als get the following homotopy pullback square:
\begin{equation}\label{hopullbacklog} \xymatrix{
L_{K(1)\vee K(2)}gl_1 Tmf_1(N) \ar[d] \ar[r]^{\:\:\ell_{tmf}} & Tmf_1(N)\widehat{_P} \ar[d] \\
L_{K(1)}Tmf_1(N) \ar[r]^{1-\langle p \rangle U^{top}} & L_{K(1)}Tmf_1(N)
}.
\end{equation}

The upper horizontal map $\ell_{tmf}$ (or sometimes its precomposition with the map $gl_1Tmf_1(N) \rightarrow L_{K(1)\vee K(2)}gl_1 Tmf_1(N)$) is called the topological logarithm  and its existence is unique to the situation of topological modular forms. Its fiber, or more precise the homotopy groups of its fiber, are closely related to open questions in number theory and it is conjectured that they are related to exotic structures on the free loop spaces $LS^n$ on spheres. For more on this we refer the reader to the discussion at the end of Chapter 10 of \cite{tmfbook}.
\end{remark}

\subsection{The Hopkins-Lawson Obstruction Theory}\label{lifting}
In \cite{hopkinslawson} the authors, based on earlier work by Aarone and Lesh, gave another approach to complex $\mathbb{E}_{\infty}$-orientations without usage of the logarithmic cohomology operations. The following two propositions summarize the results of \cite{hopkinslawson} we need.
\begin{prop}\label{hlprop1}
There exists a sequence of connective spectra $$ \ast \simeq x_0 \rightarrow x_1 \rightarrow x_2 \rightarrow x_3 \ldots$$ with the following properties:
\begin{enumerate}
    \item $x_1 \simeq \Sigma^{\infty} \mathbb{CP}^{\infty}$
    \item  $\emph{\text{hocolim}} \: x_i \simeq bu $. 
    \item Let $f_m$ be the homotopy fiber of the map $x_{m-1} \rightarrow x_m$. Then 
    \begin{enumerate}
        \item $f_m$ is $(2m-2)$-connective;
        \item the map $f_m \rightarrow \ast$ is an isomorphism in rational homology for $m>1$, an isomorphism in $p$-local homology if $m$ is not a power of $p$ and an isomorphism in $K(n)$-homology if $m>p^n$.
    \end{enumerate}
\end{enumerate}
\end{prop}
\begin{prop}\label{hlprop2}
Let $MX_{p^n}$ be the Thom spectrum corresponding to the composite $$\Sigma^{-1} x_{p^n} \rightarrow \Sigma^{-1} bu \rightarrow gl_1 \mathbb{S}.$$ We then have the following: 
    \begin{enumerate} 
    \item Let $E$ be an $\mathbb{E}_{\infty}$ ring spectrum. We have an equivalence 
    $$\pi_0 Map_{\mathbb{E}_{\infty}}( MX_1, E) \cong Map_{RingSpectra}(MU,E)$$ between $\mathbb{E}_{\infty}$-orientations by the spectrum $MX_1$ and ordinary homotopy commutative complex orientations. 
        \item We have equivalences $$ L_{K(n)} MX_{p^n} \simeq L_{K(n)} MU $$ and
        $$ L_{E(n)}MX_{p^n} \simeq L_{E(n)} MU$$ for $n\geq 0$. 
    \end{enumerate}
\end{prop}
These two propositions imply equivalences 
$$ Map_{\mathbb{E}_{\infty}}(MU, Tmf_1(N)\widehat{_p})\simeq Map_{\mathbb{E}_{\infty}}(MX_{p^2}, Tmf_1(N)\widehat{_p}) \simeq Orient(x_{p^2},gl_1 Tmf_1(N)\widehat{_p})   
$$
and 
$$ Orient(bu, L_{K(1)\vee K(2)} gl_1 Tmf_1(N)\widehat{_p}) \simeq Orient(x_{p^2},L_{K(1)\vee K(2)} gl_1 Tmf_1(N)\widehat{_p}).$$

Recall from Section \ref{orientations}  that the functor $Orient(bg,-)$ from spectra under 
$gl_1 \mathbb{S}$ to spaces commutes with homotopy pullback diagrams. We will now apply this fact to the cofiber sequence
$$
gl_1 Tmf_1(N) \rightarrow L_{K(1)\vee K(2)} gl_1 Tmf_1(N) \rightarrow \Sigma D.
$$
The cofiber $\Sigma D$ becomes a space under $gl_1 \mathbb{S}$ via the composite map
$$
gl_1 \mathbb{S} \rightarrow gl_1 Tmf_1(N)\widehat{_p} \rightarrow L_{K(1)\vee K(2)} gl_1 Tmf_1(N) \rightarrow \Sigma D 
$$
and we get the following homotopy pullback diagram

$$
\xymatrix{
Orient(x_{p^2}, gl_1 Tmf_1(N)\widehat{_p})\ar[r]\ar[d] & Orient(x_{p^2}, L_{K(1)\vee K(2)} gl_1 Tmf_1(N)\widehat{_p}) \ar[d] \\
\ast \ar[r] & Orient(x_{p^2}, \Sigma D)
}.
$$
Observe that the space $Orient(x_{p^2},\Sigma D)$ is always non-empty and the torsor $$\pi_0 Orient(x_{p^2},\Sigma D) \cong[x_{p^2},\Sigma D]$$ is canonically trivialized by the zero map $gl_1\mathbb{S}/(\Sigma^{-1}x_{p^2}) \xrightarrow{0} \Sigma D$. Hence the path components of the space $Map_{\mathbb{E}_{\infty}}(MU,Tmf_1(N)\widehat{_p})$ correspond to the maps in $Orient(x_{p^2},L_{K(1)\vee K(2)} gl_1 Tmf_1(N)\widehat{_p})$ which are sent to the trivial map in 
$ Orient(x_{p^2},\Sigma D)$. 
\begin{remark}
Analogous statements hold if we replace in the discussion above the spectra $x_{p^2}$ with $x_p$ and $x_1$, respectively.  
\end{remark}
Before we proceed further we need the following result which is due to Wilson.
\begin{prop}[\cite{Dylan}(Corollary 4.8)]\label{torsionfreed2}
Let $D$ be as above. Then $\pi_3 D$ is torsion free. 
\end{prop}

\begin{remark}
Wilson states this result only for $p>3$ but for all level structures $\Gamma$. This is due to the fact that for the levels $\Gamma_0(N)$ and $p=2,3$ 
$$ \pi_3 gl_1 Tmf(\Gamma_0(N))\widehat{_p} 
$$
may be nontrivial. However for the levels $\Gamma_1(N)$ we are interested in, these groups vanish and the argument works for all primes. 
\end{remark}
We are now in the position to state the main result of this section.
\criterion*
\begin{proof}
First note that Conditions 1. - 3. characterize the elements in $$\pi_0 Orient(bu, L_{K(1)\vee K(2)}gl_1 Tmf_1(N)\widehat{_p}) \cong \pi_0 Orient(x_{p^2}, L_{K(1)\vee K(2)} gl_1 Tmf_1(N)\widehat{_p}) .$$

It is clear that an element $f\in Orient(x_{p^i}, L_{K(1)\vee K(2)} gl_1 Tmf_1(N)\widehat{_p})$ (for $i=1,2$) gives rise to an element $\widetilde{f} \in Orient(x_{p^{i-1}}, L_{K(1)\vee K(2)}gl_1 Tmf_1(N))$ by precomposition with the map $gl_1 \mathbb{S}/x_{p^{i-1}} \rightarrow gl_1 \mathbb{S}/x_{p^i}$. Similarily for $Orient(x_{p^i},gl_1 Tmf_1(N)\widehat{_p})$.

We now have injections $$\pi_0 Orient(x_{p^i}, \Sigma D) \cong \pi_0 Map(x_{p^i}, \Sigma D) \hookrightarrow \pi_0 Orient(x_{p^{i-1}}, \Sigma D) \cong \pi_0 Map(x_{p^{i-1}}, \Sigma D):$$
For $i=2$ and all primes $p$ and for $i=1$ for $p>2$ this is immediately clear as $f_{p^i}$, the fiber of the map $x_{p^i-1} \rightarrow x_{p^i}$, is $(2p^i-2)$ connective but $\pi_* \Sigma D =0$ for $*\geq 5$, so in this case we actually have equivalences $Map(x_{p^i},\Sigma D) \simeq Map(x_{p^{i-1}}, \Sigma D)$.
For $i=1$ and $p=2$ we still have the induced exact sequence 
$$ [\Sigma f_2, \Sigma D] \rightarrow [x_2,\Sigma D] \rightarrow [x_1, \Sigma D]. 
$$
As $\Sigma f_2$ is $3$-connective we have an isomorphism $$[\Sigma f_2, \Sigma D] \cong [H\pi_4 \Sigma f_2, H\pi_4\Sigma D] \cong \text{Hom}(\pi_4 f_2, \pi_4\Sigma D).$$ 
However, as the spectrum $f_2$ is torsion but $\pi_4 \Sigma D$ is torsion free we have $[\Sigma f_2, \Sigma D]=0$.
It follows that in order to check if an element $$g\in \pi_0 Orient (x_{p^2},L_{K(1)\vee K(2)} gl_1 Tmf_1(N)\widehat{_p})$$ maps to $0 \in \pi_0 Orient(x_{p^2},\Sigma D)$ we might just as well check if the element $\widetilde{g} \in \pi_0Orient(x_1, L_{K(1)\vee K(2)}gl_1 Tmf_1(N))$ induced from $g$ via precomposition with the maps $$gl_1\mathbb{S}/(\Sigma^{-1}x_1) \rightarrow gl_1\mathbb{S}/(\Sigma^{-1}x_p) \rightarrow gl_1\mathbb{S}/(\Sigma^{-1}x_{p^2})$$ maps to $0\in \pi_0 Orient(x_1,\Sigma D)$. But in light of Proposition \ref{hlprop2} this amounts to $\widetilde{g}$ actually corresponding to an ordinary, homotopy commutative orientation. The spectra $Tmf_1(N)$ are complex orientable, so $\pi_0 Orient(x_1,gl_1 Tmf_1(N))$ is nonempty and a torsor for the group $[x_1,gl_1 Tmf_1(N)] \simeq [\mathbb{CP}^{\infty},Gl_1 Tmf_1(N)]$. This group is torsion free, so we have a sequence of inclusions 

$$[x_1,gl_1 Tmf_1(N)\widehat{_p}] \hookrightarrow [x_1,gl_1 Tmf_1(N)\widehat{_p} \otimes \mathbb{Q}] \overset{\ell_{\mathbb{Q}}}{\cong} [x_1, Tmf_1(N)\widehat{_p} \otimes \mathbb{Q}] \hookrightarrow [x_1, L_{K(1)}Tmf_1(N)\otimes \mathbb{Q}]. $$ Finally note that 
$$ x_1 \simeq_{\mathbb{Q}} bu.
$$ Hence the image under this inclusion are exactly the sequences of classical modular forms 
$$\{g_i \}  \in \prod \overline{MF}_{i}(\Gamma_1(N))\otimes \mathbb{W} \subset \prod MF_{p,i}(\Gamma_1(N)) $$ and the result follows.
\end{proof}

\begin{remark}
Condition 4 in the theorem above is superflous for the elements $g_k$ with $k >2$ as it is known by work of Hida (\cite{hida1}, \cite{hida2}) that all eigenforms of weight $k>2$ of the $U$-operator with unit eigenvalues already arise as the $p$-stabilizations of classical modular forms.
\end{remark}

\section{Constructing Orientations}\label{constructingmeasures}
In order to give an $\mathbb{E}_{\infty}$-orientation $MU \rightarrow Tmf_1(N)\widehat{_p}$ all that is left to do is finding a sequence of $p$-adic modular forms $\{ g \}$ satisfying the properties dictated by Theorem \ref{orienttmfgammaall}. We will construct the required measure in several steps. We first recall the Eisenstein measures of \cite{katzeisenstein}. 

\begin{prop}[\cite{katzeisenstein}, Theorem 3.3.3] 
There exists a (unique) measure $H^{a,b}_{\chi}$ on $\mathbb{Z}_p$ with values in $V(\Gamma_1(N),\mathbb{W})$ defined by the formula
$$\int_{\mathbb{Z}_p} x^k dH^{a,b}_{\chi} \overset{\text{def}}{=} G_{k+1}^{\chi} -\langle a,b \rangle G_{k+1}^{\chi}.
$$
\end{prop}
\begin{remark}
Katz uses  a slightly different set of Eisenstein series in \cite{katzeisenstein}. His proof however applies to preceding proposition.
\end{remark}
We are interested in measures defined on $\mathbb{Z}_p^{\times}$ rather than on the whole of $\mathbb{Z}_p$. The following result is stated in \cite{katzeisenstein} without proof.

\begin{lemma}[ see \cite{katzeisenstein} Lemma 3.5.6]\label{restrictionmeasure}
Denote with $J^{a,b}_{\chi}$ the restriction of the measure $H^{a,b}_{\chi}$ to $\mathbb{Z}_p^{\times}$. Then its moments are given by 
$$ \int_{\mathbb{Z}_p^{\times}} x^{k}dJ_{\chi}^{a,b} = (1-\langle a,b \rangle) (1-p^{k}\text{Frob})G_{k+1}^{\chi}.
$$
\end{lemma}
\begin{proof}The restriction of the functions $x^k$ to $\mathbb{Z}_p^{\times}$ (reextended to $0$ to all of $\mathbb{Z}_p$) is uniformly approximated by the functions $x^{k+(p-1)p^r}$. So we have 
$$ \int_{\mathbb{Z}_p^{\times}} x^k dJ_{\chi}^{a,b} = \varinjlim_{ r \to \infty} \int_{\mathbb{Z}_p} x^{k+(p-1)p^r} dH^{a,b}_{\chi} = (1- \langle a,b \rangle) \varinjlim_{r \to \infty} G_{k+1 + (p-1)p^r}^{\chi}.
$$ We may check the last limit on $q$-expansions and thereby get the expression
$$ \varinjlim_{r \to \infty} G_{k+1 + (p-1)p^r}^{\chi} (\text{Tate}(q),\varphi_{can},\xi) =
\varinjlim_{r \to \infty}(
\sum_{n\geq 1} \chi(n) \cdot n^{k+(p-1)p^r} + \sum_{n \geq 1} q^n \sum_{d|n} d^{k+(p-1)p^r}\chi(d)).
$$
Observe that if $p\nmid d$ then $d$ is a $p$-adic unit and we have $$\varinjlim_{r \to \infty} d^{k+(p-1)p^r} = d^k$$ and if $p|d$  we have 
$$ \varinjlim_{r \to \infty} d^{k + (p-1)p^r} =0.
$$
Hence
\begin{align*}
     \varinjlim_{r \to \infty} G^{\chi}_{k+1+(p-1)p^r}(\text{Tate}(q),\varphi_{can},\xi) = & \sum_{p\nmid n}\chi(n)n^k + \sum_{n\geq 1}q^n\sum_{\overset{d|n}{p\nmid}d} d^{k} \chi(d) \\
    = & (\sum_{n\geq 1} \chi(n) n^k + \sum_{n\geq 1} q^n \sum_{d|n} d^k \chi(d)) \\
    & - (p^k\sum_{n\geq 1} \chi(pn)n^k + \sum_{n\geq 1} q^{np} \sum_{d|n} (pd)^k \chi(pd)\\ 
    = & (1-p^k \text{Frob})G^{\chi}_{k+1}(\text{Tate}(q),\varphi_{can},\xi). 
\end{align*}
The result now follows from the $q$-expansion principle.
\end{proof}

Observe that the measure above can not quite be the one we are looking for as its $k$-th moment $G^{\chi}_{k+1}$ is not a (generalized) modular form of weight $k$ but rather one of weight $k+1$. This is however not a problem as the assignment
\begin{align*}
 Cts(\mathbb{Z}_p^{\times},\mathbb{Z}_p) \rightarrow & \: Cts(\mathbb{Z}_p^{\times},\mathbb{Z}_p)\\
 f(x) \mapsto & \:f(x)\cdot x^{-1}
\end{align*} 
defines a linear bijection. 
\begin{lemma}
There exists a (unique) measure $\widetilde{J}^{a,b}_{\chi}$ on $\mathbb{Z}_p^{\times}$ with values in $V(\mathbb{W}, \Gamma_1(N))$ defined by the formula
$$ \int_{\mathbb{Z}_p^{\times}} x^k d \widetilde{J}^{a,b}_{\chi} = (1-\langle a , b \rangle)(1-p^{k-1}\text{Frob})G_{k}^{\chi}.
$$
\end{lemma}
\begin{proof}
For $f(x) \in Cts(\mathbb{Z}_p^{\times}, \mathbb{Z}_p)$ set 
$$ \int_{\mathbb{Z}_p^{\times}} f(x) d \widetilde{J}^{a,b}_{\chi} \overset{\text{def}}{=} 
\int_{\mathbb{Z}_p^{\times}} f(x)\cdot x^{-1} d J^{a,b}_{\chi}.
$$
\end{proof}
If we set $b=1$ and $a=c$ for $c\in \mathbb{Z}_p^{\times}$ the measures $\widetilde{J}^{c,1}_{\chi}$ seem good candidates for our purpose. To check if they are indeed the correct choices we have to calculate the means
$$ \int_{\mathbb{Z})p^{\times}} 1 d \widetilde{J}^{c,1}
$$
of these measures. In order to do so we first need to recall some facts about generalized Bernoulli numbers. Let $\chi$ be a Dirichlet character modulo $N$. 
Set
$$ \sum_{r=1}^N \chi(r) \frac{t e^{rt}}{ e^{Nt} -1} = \sum_{n=1}^{\infty}B_n^{\chi} \frac{t^n}{n!}.
$$
The coefficients $B_n^{\chi}$ on the right hand side are called generalized Bernoulli numbers. \begin{remark}For the trivial character for $\chi=1$ the $B_n^{\chi}$ reduce to the ordinary Bernoulli numbers.
\end{remark}
\begin{prop}[\cite{carlitz} Theorem 4]
Let $\chi$ be a nontrivial Dirichlet character modulo $N$. If $p^k| n$ but $p \nmid N$, then $p^k$ divides the numerator of $B_n^{\chi}$. 
\end{prop}
\begin{coro}\label{limitlfunction}
Let $\chi$ be a nontrivial Dirichlet character modulo $N$ and $p \nmid N$. Then we have for the corresponding $L$-function 
$$ L(1-(p-1)p^r, \chi) \equiv  0 \: \emph{\text{mod}} \: \mathbb{Z}_p[\xi].
$$
\end{coro}
\begin{proof}
This immediately follows from the preceding proposition together with the well-known relationship 
$$ -L(1-n,\chi) = \frac{B_n^{\chi}}{n} 
$$
between generalized Bernoulli numbers and $L$-functions (see \cite{leopoldt}).
\end{proof}
\begin{prop}\label{measurezero}
Let $b_n$ be a series congruent to $L(1-n,\chi)$ modulo $\mathbb{Z}_p[\xi]$ and $c\in \mathbb{Z}_p^{\times}$. Then
$$ \lim_{r \to \infty} (1-c^{(p-1)p^r})b_{(p-1)p^r} =0.
$$
\end{prop}
\begin{proof}
For $c$ a $p$-adic unit we have $$
1- c^{(p-1)p^r} \equiv 0 \: \text{mod}\: p^r.
$$
By Corollary \ref{limitlfunction} and our assumption on the series $b_n$ we get
$$ (1-c^{(p-1)p^r})b_{(p-1)p^r} \equiv 0 \: \text{mod} \: p^r.
$$
The result now follows by passing to the limit.
\end{proof}
\begin{theo}\label{meanzero}
For all $c\in \mathbb{Z}_p^{\times}$ 
$$ \int_{\mathbb{Z}_p^{\times}}1 d\widetilde{J}^{c,1} = 0.
$$
\end{theo}
\begin{proof}
Recall that we have $1= x^0 = \lim_{r \to \infty} x^{(p-1)p^r}$. Hence we have to calculate the limit 
$$ \lim_{r\to \infty} \int_{\mathbb{Z}_p^{\times}} x^{(p-1)p^r} d \widetilde{J}^{c,1}_{\chi}=
\lim_{r\to \infty} (1-c^{(p-1)p^r})(1-p^{((p-1)p^r -1)} \text{Frob})G_{(p-1)p^r}^{\chi}.
$$
The result now follows from Proposition \ref{measurezero} and the $q$-expansion principle as the $q$-expansion of $(1-p^{k-1}\text{Frob})G_k^{\chi}$ is congruent to $L(1-k,\chi)$ modulo $\mathbb{Z}_p[\xi]$.
\end{proof}
We see that the family of measures $\widetilde{J}^{c,1}$ can not quite be the ones that we are looking for as their means are $0$ and not $\frac{1}{p}\log(c^{p-1})$ as required. The remedy however is easy: In \cite{AHR}(Proposition 10.10) the authors construct for each $c\in \mathbb{Z}_p^{\times}/\{\pm 1\}$ a measure $\nu_c$ with values in the ring of $p$-adic modular forms of level $1$ uniquely defined by the formula
$$\int_{\mathbb{Z}_p^{\times} /\{ \pm 1 \}}x^k = (1-c^k)(1-p^{k-1}\text{Frob})G_k
$$
where $G_k$ are the classical Eisenstein series for the trivial Dirichlet character $\chi =1$. Moreover they show that
$$ \int_{\mathbb{Z}_p^{\times}/\{\pm 1 \}} 1 d\nu_c = \frac{1}{2p} \log(c^{p-1}). 
$$
These measures extend uniquely to measures $\overline{\nu}_c$ defined on the whole of $\mathbb{Z}_p^{\times}$ by the formula 
$$ \int_{\mathbb{Z}_p^{\times}} x^k d\overline{\nu}_c = 2 \int_{\mathbb{Z}_p^{\times}/\{ \pm 1 \}} x^k d\nu_c
$$
 as the moments of the measures $\nu_c$ are zero for $k$ odd and $k\leq 4$. Their means are therefore given by 
\begin{equation}\label{meansnu}
    \int_{\mathbb{Z}_p^{\times}} 1 d \overline{\nu}_c = \frac{1}{p} \log(c^{p-1}).
\end{equation}  

\begin{theo}
Let $c$  be a $p$-adic unit and $\chi$ a nontrivial Dirichlet character modulo $N$ such that $\chi(-1)=-1$. Denote with $\mu^{\chi}_c$  the measures on $\mathbb{Z}_p^{\times}$ with values in $ MF_{p,*}(\Gamma_1(N))$ defined by 
$$ \mu_{c}^{\chi} = \overline{\nu}_c + \widetilde{J}_{\chi}^{c,1}. 
$$
Then the $\mu_c^{\chi}$ correspond to an element in $\pi_0 Map_{\mathbb{E}_{\infty}}(MU,Tmf_1(N)\widehat{_p})$.

\end{theo}
\begin{proof}
First observe that the measures $\widetilde{J}^{c,1}$ already take values in the subring $MF_{p,*}(\mathbb{W},\Gamma_1(N))$ of $V(\mathbb{W},\Gamma_1(N))$ and that the ring $MF_{p,*}$ of holomorphic $p$-adic modular forms of level $1$ is a subring of the ring $MF_{p,*}( \Gamma_1(N),\mathbb{W})$. Hence, as the sum of two measures is again a measure, the expression $ \mu_c^{\chi}=\overline{\nu}_c + \widetilde{J}^{c,1}$ indeed defines a measure on $\mathbb{Z}_p^{\times}$ with values in $MF_{p,*}(\mathbb{W},\Gamma_1(N))$.  From Theorem \ref{meanzero} and Equation \ref{meansnu} we see that 
$$ \int_{\mathbb{Z}_p^{\times}} 1 d \mu^{\chi}_c = \frac{1}{p}\log(c^{p-1}) 
$$
hence, after Theorem \ref{k1tmforient}, defines an element in $Orient(bu, L_{K(1)} gl_1 Tmf_1(N)\widehat{_p})$. 

The discussion in \ref{eisensteinseries} showed that the Eisenstein series 
$G_k^{\chi}$ are eigenforms for the Hecke operators $T_p$ for the eigenvalues $(1-p^{k-1}\chi(p))$. Together with \cite{AHR}(Remark 12.2) we see that the element defined by the $\mu_c^{\chi}$ lifts to an element in $Orient(bu, L_{K(1)\vee K(2)} gl_1 Tmf_1(N)\widehat{_p})$. 

Finally observe that $$
\int_{\mathbb{Z}_p^{\times}} x d \mu_c^{\chi} = (1- \:\text{Frob})G_1^{\chi}
$$
and $G_1^{\chi} \in MF_1(\Gamma_1(N)) \otimes \mathbb{Z}_p$ is already a classical modular form, hence by Theorem \ref{orienttmfgammaall} we obtain a lift to an element in $Map_{\mathbb{E}_{\infty}}(MU, Tmf_1(N)\widehat{_p})$.

\end{proof}
%\begin{theo}[\cite{katzeisenstein}\label{eisensteinmeasure} 3.4.1, 3.5] Let $G_{k,\chi}$ be a %Eisenstein series as defined in section \ref{eisensteinseries}.
%There exists measures $J^{a,b}$ on 
%$\mathbb{Z}_p^{\times} \times \mathbb{Z/N}^{\times}$  with values in $V(\Gamma, %\mathbb{Z}_p[\xi_N])$ whose moments are given by
%$$\int_{\mathbb{Z}_p^{\times} \times \mathbb{Z}/N^{\times}} x^k dJ^{a,b} = (1-\langle a, b %\rangle) (1 - p^{k-1}\text{Frob })G_{k,\chi}.
%$$

%\end{theo}

%We set $a=c$ and $b=1$ in the theorem above and restrict this measure to %$\mathbb{Z}_p^{\times}$ to get a candidate for our desired measure. 
%\begin{remark}
%Katz actually proofs the existence of such a measure for a different set of Eisenstein series, %but his arguments carry over to our case verbatim.
%\end{remark}

Using the homtopy pullback square

$$\xymatrix{ Map_{\mathbb{E}_{\infty}}(MU, Tmf_1(N)) \ar[r] \ar[d] &   \prod_p Map_{\mathbb{E}_{\infty}}(MU, Tmf_1(N)\widehat{_p}) \ar[d] \\
Map_{\mathbb{E}_{\infty}}(MU,Tmf_1(N)\otimes \mathbb{Q}) \ar[r] & \prod_p Map_{\mathbb{E}_{\infty}}(MU, Tmf_1(N)\widehat{_p}\otimes \mathbb{Q} )
}
$$
and observing that we have a pullback square after taking connected components ($\pi_1$ of the lower right corner vanishes) we can now assemble these orientations to an integral orientation.
\orientations*
%\begin{theo}\label{msuorientations}
%Let $\chi$ be a nontrivial Dirichlet character modulo $N$ such that $\chi(-1)=-1$ and $G_{k}^{\chi}$ the corresponding Eisenstein series.
%Then there exists an $\mathbb{E}_{\infty}$-orientation 
%$$g_{eisen}:MU[\frac{1}{N},\xi_N] \rightarrow Tmf_1(N)
%$$
%with characteristic series given by
%$$\frac{u}{\exp(u)_{eisen}} = \exp (\sum_{k \geq 2} (G_k + G_{k}^{\chi})\frac{u^k}{k!}).
%$$
%Furthermore if $f:MU[\frac{1}{N},\xi_N] \rightarrow Tmf_1(N) $ is any other $\mathbb{E}_{\infty}$-orientation  with %characteristic series $\exp (\sum_{k\geq 2} F_k \frac{u^k}{k!})$ then
%$$F_k \equiv (G_k +G_k^{\chi}) \text{ mod }\: \overline{MF_*}(\Gamma_1(N),\mathbb{Z}[\frac{1}{N},\xi_N]). $$
%\end{theo}
\begin{remark}\label{conundrum}
We would now like to justify our insistence on adjoining $N$-th roots of unity everywhere in sight. It turns out that one can (and this is exactly what Katz does in \cite{katzeisenstein} and \cite{katzeisenstein2}) construct measures $H^{a,b}$ on $\mathbb{Z}_p$ whose moments are given by 
$$\int_{\mathbb{Z}_p} x^kdH^{a,b} = (1 -\langle a,b) \rangle G_{f,k+1}
$$
where the $G_{f,k+1}$ are a variation of the Eisenstein series that we used and which are already modular forms in $MF_{p,k+1}(\Gamma_1(N)^{arith},\mathbb{Z}_p)$ (so it seems we do not need to adjoin roots). 
However, when restricting these measures to $\mathbb{Z}_p^{\times}$ it turns out that the moments are now given by
$$ (1-\langle a,b)(1-p^k\langle p \rangle \text{Frob}^{arith})G_{f,k+1}
$$
(the reader might want to have a look at Remark \ref{frobeniuslevelstructures} in Section \ref{atkinsoperator} again). In \cite{katzeisenstein} Katz does not distinguish in his notation between arithmetic and naive level structures \footnote{he does however in the companion paper \cite{katzeisenstein2} where the factor $\langle p \rangle$ appears} and his formula for the moments of the restrictions to $\mathbb{Z}_p^{\times}$ reads as (\cite{katzeisenstein} Lemma 3.5.6) 
$$(1-\langle a,b \rangle) (1-p^{k}\text{Frob}) G_{f,k}.
$$
We would like to stress that there is no mathematical error on Katz's part only that his notation is probably slightly misleading as the operation that he calls "$\text{Frob}$" (this is the operation we called "$\widetilde{\text{Frob}}$" in Section \ref{atkinsoperator}) does \textbf{not} reduce to the $p$-th power operation modulo $p$ on $p$-adic modular forms with arithmetic $\Gamma_1(N)$-level structures (only up to an action by $\langle p \rangle$).

On the other hand, in our application we require that the operator $(1- p^{k-1} \langle p \rangle \text{Frob}^{arith})$ corresponds to the effect on homotopy of the $K(1)$-local logarithm $\ell_1$ (see Section \ref{k1locallog}) for $L_{K(1)}Tmf_1(N)$, i.e., the operator $\langle p \rangle \text{Frob}^{arith}$ should exactly reduce to the $p$-th power map modulo $p$ in order to be compatible with a $\theta$-algebra structure.

A close examination of the proof of Lemma \ref{restrictionmeasure} reveals that the diamond operators $\langle p \rangle$ inevitably make an appearance (unless of course one restricts oneself to $p$-adic modular forms where $\langle p \rangle$ acts trivially, i.e those of level $\Gamma_0(N)$) so one will always run into this conundrum when one insists on working over $\mathbb{Z}_p$. 
But, as we have seen, everything works out when working over $\mathbb{W}$ instead - this is an artefact of the fact that after adjoining an $N$-th root of unity the notions of arithmetic and naive level structures coincide.

\end{remark}
\bibliographystyle{amsalpha}

\bibliography{bibliography.bib}
\Addresses
\end{document}